\documentclass[12pt]{amsart}

\usepackage[margin = 1in]{geometry}
\usepackage{amsfonts, amsmath,amscd, amssymb, cite, color, latexsym, mathrsfs, 
slashed, stmaryrd, verbatim, wasysym , tensor }
\usepackage[all]{xy}
\usepackage{graphicx}
\usepackage{hyperref}
\usepackage[applemac]{inputenc}

\newtheorem*{acknowledgements}{Acknowledgements}
\newtheorem{theorem}{Theorem}[section]
\newtheorem{corollary}[theorem]{Corollary}
\newtheorem{definition}[theorem]{Definition}
\newtheorem{example}[theorem]{Example}
\newtheorem{lemma}[theorem]{Lemma}
\newtheorem{proposition}[theorem]{Proposition}
\newtheorem{remark}[theorem]{Remark}

\newtheorem{conjecture}[theorem]{Conjecture}

\numberwithin{equation}{section}

\let\oldsqrt\sqrt
\def\sqrt{\mathpalette\DHLhksqrt}
\def\DHLhksqrt#1#2{%
\setbox0=\hbox{$#1\oldsqrt{#2\,}$}\dimen0=\ht0
\advance\dimen0-0.2\ht0
\setbox2=\hbox{\vrule height\ht0 depth -\dimen0}%
{\box0\lower0.4pt\box2}}

\renewcommand{\tilde}{\widetilde}
\renewcommand{\bar}{\overline}
\renewcommand{\Re}{\operatorname{Re}}

\newcommand\pa{\partial}

\newcommand\cf{cf\@. }

\newcommand\eps\varepsilon
\renewcommand\epsilon\varepsilon

\newcommand\Id{\operatorname{Id}}
\renewcommand\Im{\operatorname{Im}}

\newcommand\rank{\operatorname{rank}}

\renewcommand\Re{\operatorname{Re}}
\DeclareMathOperator*\Res{\operatorname{Res}}

\newcommand\pr{\operatorname{pr}}

\newcommand\GL{\operatorname{GL}}

\newcommand\Ad{\operatorname{Ad}}

\newcommand\Symm{\operatorname{Sym}}
\newcommand\bL{\overline{L}}

\newcommand\paperintro%
        {%
         }
\newcommand\paperbody%
        {%
         }

\newcommand{\C}{\mathbb{C}}
\newcommand{\R}{\mathbb{R}}

\newcommand{\Z}{\mathbb{Z}}
\newcommand{\Q}{\mathbb{Q}}
\newcommand{\bH}{\mathbb{H}}
\newcommand{\N}{\mathbb{N}}

\renewcommand{\Re}{\operatorname{Re}}
\renewcommand{\Im}{\operatorname{Im}}
\newcommand{\bsl}{\backslash}

\newcommand{\G}{\mathbf{G}}
\newcommand{\T}{\mathbf{T}}

\newcommand\bbC{\mathbb{C}}

\newcommand\bbH{\mathbb{H}}

\newcommand\bbN{\mathbb{N}}

\newcommand\bbP{\mathbb{P}}
\newcommand\bbQ{\mathbb{Q}}
\newcommand\bbR{\mathbb{R}}
\newcommand\bbS{\mathbb{S}}

\newcommand\bbZ{\mathbb{Z}}

\newcommand\cH{\mathcal{H}}

\newcommand\cK{\mathcal{K}}
\newcommand\cL{\mathcal{L}}

\newcommand\cO{\mathcal{O}}

\newcommand\SL{\operatorname{SL}}
\newcommand\SO{\operatorname{SO}}

\newcommand\Spin{\operatorname{Spin}}
\newcommand\bX{\overline{X}}

\newcommand\tE{\widetilde{E}}

\newcommand\SU{\operatorname{SU}}
\newcommand\bz{\overline{z}}
\newcommand\vol{\operatorname{vol}}
\newcommand\free{\operatorname{free}}
\newcommand\tor{\operatorname{tor}}
\newcommand\bE{\overline{E}}
\newcommand\bV{\overline{V}}
\newcommand\bmu{\overline{\mu}}
\newcommand\ga{\mathfrak{a}}
\newcommand\gf{\mathfrak{g}}
\newcommand\kf{\mathfrak{k}}
\newcommand\pf{\mathfrak{p}}

\DeclareMathAlphabet{\mathpzc}{OT1}{pzc}{m}{it}

\begin{document}

\title[Exponential growth torsion]{Exponential growth of torsion in the 
cohomology of arithmetic hyperbolic manifolds}

\author{Werner M\"uller}
\address{Universit\"at Bonn, Mathematisches Institut, Endnicher Allee 60, D-53115 Bonn, Germany}
\email{mueller@math.uni-bonn.de}
\author{Fr\'ed\'eric Rochon}
\address{D\'epartement de Math\'ematiques, UQ\`AM}
\email{rochon.frederic@uqam.ca }

\begin{abstract}
For $d=2n+1$ a positive odd integer, we consider sequences of arithmetic subgroups of $\SO_0(d,1)$ and $\Spin(d,1)$ yielding corresponding hyperbolic manifolds of finite volume and show that, under appropriate and natural assumptions, the torsion of the associated cohomology groups grows exponentially.     
\end{abstract}

\maketitle

\tableofcontents

\section{Introduction}

The main goal of this paper is to extend the results of Bergeron and Venkatesh
\cite{BV} on the growth of torsion in the cohomology of cocompact arithmetic
groups to the case of arithmetic hyperbolic manifolds of finite volume. To
begin with we recall some of the results of \cite{BV}.
Let $\G$ be a connected semi-simple algebraic group defined over $\Q$ and let
$G=\G(\R)$ be its group of real points. Let $K$ be a maximal compact subgroup
of $G$ and $\widetilde X=G/K$ the associated Riemannian symmetric space. 
Let $\Gamma\subset \G(\Q)$ be an arithmetic subgroup. Assume that $\Gamma$ is
cocompact or equivalently that $\G$ is anisotropic over $\Q$. Let  
$X:=\Gamma\bsl \widetilde X$ be the corresponding locally symmetric space. 
Let $V_\Z$ be an arithmetic $\Gamma$-module in the sense of \cite{BV}. 
This means that there exists a $\Q$-rational representation $\varrho$ of $\G$ 
on a finite-dimensional $\Q$-vector space $V_\Q$ such that $V_\Z$ is a 
$\Z$-lattice in $V_\Q$ which is invariant under $\varrho(\Gamma)$. The cohomology
$H^*(\Gamma;V_\Z)$ of $\Gamma$ with coefficients in $V_\Z$ is a finitely 
generated abelian group. If $\Gamma$ is torsion free,  
$H^*(\Gamma;V_\Z)$ is isomorphic to the cohomology $H^*(X;E_\Z)$ of $X$ 
with coefficients in the local system $E_\Z$ associated to the $\Gamma$-module 
$V_\Z$. 

Let $\gf$ and $\kf$ be the Lie algebras of $G$ and $K$, respectively. Let 
$\delta(\widetilde X):=\rank(\gf_\C)-\rank(\kf_\C)$ be the fundamental rank.
As explained by Bergeron and Venkatesh in \cite{BV}, for arithmetic reasons, 
when $\delta(\widetilde X)=1$, one expects that $H^*(\Gamma;V_{\bbZ})$ should 
have a lot of torsion and a small free part.  There are various manifestations 
of this 
phenomenon.  One of them, proved in \cite{BV} by Bergeron and Venkatesh, is 
that if $\Gamma\supset\cdots\supset\Gamma_{N-1}\supset\Gamma_N\supset\cdots
\supset\{e\}$  is a decreasing sequence of congruence subgroups such that 
$\displaystyle \cap_{N} \Gamma_N=\{1\}$, and $V_{\bbZ}$ is a {\it strongly acyclic $\Gamma$-module}, meaning that the Laplacians on $V_\Z\otimes\C$-valued 
$i$-forms on $\Gamma_N\backslash\widetilde X$ have a uniform spectral gap at 
$0$ for all $i$ and $N$, then
\begin{equation}
  \liminf_{N\to \infty} \sum_j \frac{\log|H^j_{\tor}(\Gamma_N;V_{\bbZ})|}
{[\Gamma: \Gamma_N]} \ge c_{G,V_{\bbZ}}>0,
\label{int.1}\end{equation}
where the sum is over the integers $j$ such that $j+\frac{\dim(G/S)-1}2$ is odd.  If one fixes $\Gamma$ and varies the arithmetic $\Gamma$-module $V_{\bbZ}$, 
parallel results were obtained in \cite{Marshall-Muller, MP2014b}.  In both 
settings, the strategy of the proof of these results is to use 
\cite{Muller1993} (see also \cite{Bismut-Zhang}) and to compute the limiting 
behavior of analytic torsion using spectral methods, see in particular 
\cite{Muller2012, MP2013} for sequences of $\Gamma$-modules. Since many 
important arithmetic groups are not cocompact, it is very desirable to extend 
these results to the finite volume case.  

The aim of this paper is to establish an analog of \eqref{int.1} for 
odd-dimensional hyperbolic manifolds of finite volume. So assume that $\G$  
is an algebraic group over $\Q$ such that $G=\G(\R)$ equals $\SO(d,1)$ or 
$\Spin(d,1)$ with 
$d=2n+1$ a positive odd integer. Fix a maximal compact subgroup $K$ which 
equals $\SO(d)$ or $\Spin(d)$, respectively. Let $\Gamma\subset \G(\bbQ)$ be an 
arithmetic subgroup. We are interested in the case where $\Gamma$ is not 
cocompact. Therefore we assume that $\G$ is not anisotropic.
Then  $X:= \Gamma\setminus G/K$ is a non-compact finite volume
hyperbolic orbifold, possibly with orbifold singularities since we are not assuming that $\Gamma$ is torsion-free.
When $n=1$ and $G=\Spin(3,1)\cong \SL(2,\bbC)$, 
generalizations of \cite{BV} and  \cite{Marshall-Muller} have been already 
obtained.  In \cite{Pfaff2014b}, Pfaff has obtained an analog of \eqref{int.1} 
when $\{\Gamma_N\}$ is a sequence of congruence subgroups of a torsion-free 
congruence subgroup $\Gamma$ of a Bianchi group.  Similarly, for a 
torsion-free congruence subgroup $\Gamma$ of a Bianchi group, Pfaff and 
Raimbault in \cite{Pfaff-Raimbault} have obtained an analog of 
\cite{Marshall-Muller}.  

Next we introduce appropriate $\Gamma$-modules. 
Let ${\boldsymbol \varrho}$ be  a $\Q$-rational representation of 
$\G$ on a finite-dimensional $\Q$-vector space $V_\Q$. Assume that 
there exists
a lattice $V_\Z\subset V_\Q$, i.e., $V_\Q=V_\Z\otimes_\Z\Q$, which is invariant 
under $\Gamma$. Furthermore, assume that representation
$\varrho$
of $G$ on $V=V_\Q\otimes_{\bbQ} \bbR$, which is associated to ${\boldsymbol \varrho}$,
 is a direct sum of finitely many  
irreducible finite dimensional complex representations $\varrho_i$  each of 
which 
satisfies
\begin{equation}
\varrho_i\circ \theta\ne \varrho_i,
\label{int.3}\end{equation}
where $\theta$ is the standard Cartan involution with respect $K$. The 
existence of such $\Gamma$-modules is proved in \cite[sect. 8.1]{BV}. See
section \ref{co.0} for more details.
Let $\{\Gamma_i\}$ be a sequence of torsion-free finite index subgroups of 
$\Gamma$. Then
$$
             X_i= \Gamma_i\setminus G/K
$$
is a $d$-dimensional oriented hyperbolic manifold of finite volume. Denote by
$\bar X_i$ the Borel-Serre compactification of $X_i$. This is a compact 
manifold with boundary which consists of the disjoint union of finitely many
tori. For each $\Gamma_i$, assume that for each $\Gamma_i$-cuspidal parabolic subgroup $P$, we have
\begin{equation}
       \Gamma_i\cap P= \Gamma_i\cap N_P,
\label{int.2}\end{equation}
where $N_P$ denotes the nilpotent radical of $P$. The assumption \eqref{int.3} 
implies that all Laplacians on $V_\Z\otimes\C$-valued $q$-forms on 
$X_i=\Gamma_i\backslash\bH^d$ have a uniform spectral gap at $0$ for all $q$ and
$i$. In particular, the spaces of $L^2$-harmonic forms all vanish. However, 
the $\Gamma_i$-module $V_\Z$ is not acyclic. The interior cohomology vanishes,
but there is always cohomology coming from the boundary. In analogy with
\cite{BV}, we call such a $V_\Z$ a {\it strongly $L^2$-acyclic arithmetic 
$\Gamma$-module}. 

For our approach we need to work with self-dual $\Gamma$-modules. Let 
$E_\Z$ be the bundle of free $\Z$-modules over $X_i$ associated to the 
arithmetic $\Gamma_i$-module $V_\Z$. Put 
\[
\bar E_\Z=E_\Z\oplus E_\Z^*\;\text {and}\; \bar\varrho=\varrho\oplus\varrho^*,
\]
where $\varrho^*$ denotes the contragredient representation of $\varrho$. 
Let $T^{(2)}_X(\varrho)$ denote the $L^2$-torsion of $X=\Gamma\backslash\bH^d$ 
and $\varrho$ seen as a complex representation \cite{Lo}, \cite{MV}. Recall that $\log T^{(2)}_X(\varrho)=
t^{(2)}_{\bH^d}(\varrho)\cdot\vol(\Gamma\backslash\bH^d)$, 
where $t^{(2)}_{\bH^d}(\varrho)$ is a constant that depends only on $\bbH^d$ and
$\varrho$. Let $\overline{\varrho}:=\varrho\oplus \varrho^*$ and
\[
t^{(2)}_{\bbH^d}(\overline{\varrho})= t^{(2)}_{\bbH^d}(\varrho)+ 
t^{(2)}_{\bbH^d}(\varrho^*).
\]
Our main result is the following theorem.
\begin{theorem}
Suppose that the above sequence $\{\Gamma_i\}$ of subgroups of $\Gamma$ is cusp-uniform in the sense of Definition~\ref{gt.15} below. 
Assume that  $\displaystyle \lim_{i\to \infty}\ell(\Gamma_i)=\infty$, where $\ell(\Gamma_i)$ is the length of the shortest closed geodesic on $X_i$, and that
\begin{equation}
  \lim_{i\to \infty}  \frac{\kappa(X_i)+ \sum_{P\in \mathfrak{P}_{\Gamma_i}}\log [\Gamma_P: \Gamma_{P,i}]}{[\Gamma: \Gamma_i]}=0,
\label{gt.31}\end{equation}
where $\kappa(X_i)= \# \mathfrak{P}_{\Gamma_i}$ is the number of cusps of $X_i$.  
Let $V_\Z$ be a strongly $L^2$-acyclic arithmetic $\Gamma$-module and $E_\Z$ the
local system associated to $V_\Z$. 
Then  we have 
\begin{equation}
   \liminf_{i\to \infty} \frac{\underset{q+n \; \; \operatorname{odd}}{\sum}\log |H^q_{\tor}(\bX_i;\bE_{\bbZ})| }{[\Gamma: \Gamma_i]}\ge (-1)^n t^{(2)}_{\bbH^d}(\overline{\varrho}) \vol(X)>0,
\label{gt.32}\end{equation}
 Furthermore, if $E_{\bbZ}^*\cong E_{\bbZ}$ and $\varrho$ is self-dual, then in fact
\begin{equation}
 \liminf_{i\to \infty} \frac{\underset{q \; \; \operatorname{even}}{\sum}\log |H^q_{\tor}(\bX_i;E_{\bbZ})| }{[\Gamma: \Gamma_i]}\ge (-1)^n t^{(2)}_{\bbH^d}(\varrho) \vol(X)>0.
 \label{gt.32b}\end{equation}
  \label{gt.30}\end{theorem}

\begin{remark}
As explained in \cite[Theorem~1.2]{MP2014}, if the sequence $\{\Gamma_i\}$ is such that $\displaystyle \lim_{i\to \infty} [\Gamma:\Gamma_i]=\infty$, and if each $\gamma\in \Gamma\setminus \{1\}$ only belongs to finitely many $\Gamma_i$, then \eqref{gt.31} is automatically satisfied and $\displaystyle \lim_{i\to \infty} \ell(\Gamma_i)=\infty$.  In particular, this is the case if $\{\Gamma_i\}$ is a decreasing sequence 
$$
       \cdots \subset \Gamma_{i+1}\subset \Gamma_i\subset \cdots \subset \Gamma_1\subset \Gamma_0:= \Gamma
$$ 
of normal subgroups such that $\cap_i \Gamma_i= \{1\}$.  
\label{gt.37}\end{remark}
\begin{remark}
Recall that $\Gamma_i$ is torsion free and $G/K$ is contractible. Therefore we
have  
\[
H^q(\bX_i,E_{\bbZ})\cong H^q(\Gamma_i;V_\Z).
\]
See \cite[Chapt. II, Prop. 4.1]{Brown} for the case of trivial coefficients.
This gives the connection with \eqref{int.1}.
\end{remark}
One important case covered by Theorem~\ref{gt.30} is for sequences of principal congruence subgroups.  More precisely, let $a_1,...,a_d\in\N$ and let $q$ be the
quadratic form defined by 
\begin{equation}\label{quf}
q(x_1,...,x_d,x_{d+1})=a_1x_1^2+\cdots+a_dx_d^2-x_{d+1}^2.
\end{equation}
Let $\G$ be the algebraic group defined by
$q$, i.e., for every $\Q$-algebra $A$ with unit element we have 
\[
\G(A)=\{g\in\SL(n,A)\colon q(gx)=q(x),\; x\in A^{d+1}\}.
\]
Then $\G$ is defined over $\Q$ and $G:=\G(\R)\cong\SO(d,1)$. 
By \cite{BHC}, the subgroup $\Gamma:=\mathbf{G}(\bbZ)$ is a lattice in $G$.  
Furthermore, for $d\ge 4$, $\G(\bbZ)$ is noncompact \cite[Prop. 6.4.1]{Mo2015},
and for $d\neq 3,7$, every noncompact arithmetic subgroup of $\SO(d,1)$
is up to commensurability and conjugation equal to some $\G(\Z)$ \cite[Prop. 
6.4.2]{Mo2015}.

For $j\in\N$ let
\begin{equation}
   \Gamma(j):= \{ A\in \Gamma \; | \; A \equiv 1 \mod(j) \}
\label{gt.38}\end{equation}  
denote the principal congruence subgroup of level $j$, so that $\Gamma(j)$ coincide with the kernel of the canonical map $\Gamma\to \G(\bbZ/j\bbZ)$, which implies in particular that $\Gamma(j)$ is a normal subgroup of $\Gamma$.  If $j\ge 3$, we know that $\Gamma(j)$ is neat in the sense of Borel \cite{B1969}, so that $\Gamma(j)$ is torsion-free and satisfies \eqref{int.2}.  Furthermore, by \cite[Lemma~4]{DH1999}, \cf \cite[Lemma~10.1]{MP2014}, we know that the sequence $\{\Gamma(j)\}$ is cusp-uniform.  On the other hand, it is obvious that each $\gamma \in \Gamma$ belongs only to finitely many $\Gamma(j)$.  Finally, for a $\Gamma$-cuspidal parabolic subgroup $P$, \cite[Lemma~4]{DH1999} implies that  $\displaystyle \lim_{j\to \infty} [\Gamma\cap N_P: \Gamma(j)\cap N_P]=\infty$, hence that  $\displaystyle \lim_{j\to \infty}[\Gamma : \Gamma(j)] =\infty$.  
Thus, applying Theorem~\ref{gt.30} gives the following.  

\begin{corollary}
Let the assumptions be as above.   For each $j\ge 3$, let 
$X_j:= \Gamma(j)\setminus \bbH^d$ and 
denote by $E_{\bbZ}= \bbH^d\times_{\left. \varrho\right|_{\Gamma(j)}} V_{\bbZ}$ the bundle of free $\bbZ$-modules over $X_j$ associated to a strongly $L^2$-acyclic 
arithmetic $\Gamma$-module $V_{\bbZ}$.  Then 
$$
\liminf_{j\to \infty} \frac{\underset{q+n \; \; \operatorname{odd}}{\sum}\log |H^q_{\tor}(\bX_j;\bE_{\bbZ})| }{[\Gamma: \Gamma(j)]}\ge (-1)^n t^{(2)}_{\bbH^d}(\overline{\varrho}) \vol(X)>0,$$   
where $X:=\Gamma\setminus \bbH^d$.  Furthermore, if $E_{\bbZ}\cong E^*_{\bbZ}$ and $\varrho$ is self-dual, then
\begin{equation}
\liminf_{j\to \infty} \frac{\underset{q \; \; \operatorname{even}}{\sum}\log |H^q_{\tor}(\bX_j; E_{\bbZ})| }{[\Gamma : \Gamma(j)]}\ge (-1)^n t^{(2)}_{\bbH^d}(\varrho) \vol(X)>0.
\label{int.9}\end{equation}
\label{gt.39}\end{corollary}
We can give the following explicit example where the condition $E_{\bbZ}\cong E_{\bbZ}^*$ is satisfied.
\begin{example}
Let $\varrho: \SO_o(d,1)\to \Lambda^{n+1}(\bbC^{2n+2})$ be the $(n+1)$th exterior power of the standard representation.  By \cite[Theorem~5.5.13]{GW}, $\varrho=\varrho_+\oplus\varrho_-$ for two distinct irreducible complex representations such that $\varrho_+=\varrho_-\circ\vartheta$.  In particular, condition \eqref{int.3} is satisfied.  On the other hand,
$\Lambda^{n+1}(\bbC^{2n+2})= V_{\bbZ}\otimes_{\bbZ}\bbR$ for the free $\bbZ$-module 
$$
   V_{\bbZ}= \Lambda^{n+1}(\bbZ[i]^{2n+2})
$$
naturally preserved by $\Gamma= \mathbf{G}(\bbZ)$ defined in terms of the standard quadratic form of signature $(d,1)$, that is, the one with $a_1=\cdots=a_{d+1}=1$.
Now, this quadratic form induces a canonical isomorphism $V_{\bbZ}^*\cong V_{\bbZ}$, which in turn yields an isomorphism $E_{\bbZ}^*\cong E_{\bbZ}$ for the associated bundle of free $\bbZ$-modules, so that \eqref{int.9} holds in this case.  When $n$ is odd, we can look at the restriction $V_{\bbZ,\pm}$ of $V_{\bbZ}$ to $\varrho_{\pm}$ to get examples where the representation is irreducible.  Indeed, in this case, the representation $\varrho_{\pm}$ is automatically self-dual by \cite[\S~3.2.5]{GW}, while the standard quadratic form induces an isomorphism $V_{\bbZ,\pm}\cong 4(V^*_{\bbZ,\pm})$, so that the corresponding bundles of free $\bbZ$-modules are still self-dual.   
\end{example}

For $n>1$, Theorem~\ref{gt.30} and Corollary~\ref{gt.39} give the first examples of towers of non-compact odd dimensional arithmetic hyperbolic manifolds of finite volume whose torsion in cohomology grows exponentially with respect to the volume.    In contrast, when $n=1$,  some results of exponential growth were already obtained in \cite{Pfaff2014b, Pfaff-Raimbault} using arguments specific to the $3$-dimensional setting.  Still, in this case, our methods can be used  to generalize and sharpen those known results, as well as providing a more systematic proof.  To explain this in more details, take $n=1$ and $G=\Spin(3,1)\cong \SL(2,\bbC)$.  Let $\varrho_m:\SL(2,\bbC)\to \GL(V_m)$ be the $m$th symmetric power $V_m=\Symm^m(\bbC^2)$ of the standard representation of $\SL(2,\bbC)$.  Let also $F=\bbQ(\sqrt{-D})$ be an imaginary quadratic number field, where $D\in \bbN$ is square-free, and let $\cO_D$ be its ring of integers.  Denote by $\Gamma(D):=\SL(2,\cO_D)$ the corresponding Bianchi group.  In $V_m$, 
\begin{equation}
  \Lambda_m:= \Symm^m(\cO_D^2)
\label{int.5}\end{equation}
is a natural lattice preserves by $\Gamma(D)$.  For  a  non-zero ideal $\mathfrak{a}$ of $\cO_D$, let $\Gamma(\mathfrak{a})$ be the associated congruence subgroup of level $\mathfrak{a}$ as in \eqref{gt.48} below.  Let $N(\mathfrak{a})$ denote the norm of $\mathfrak{a}$.  In this context, we will see that Theorem~\ref{gt.30} implies the following.
  
\begin{corollary}
  If $\ga_i$ is a sequence of nonzero ideals in $\cO_D$ such that each $N(a_i)$ is sufficiently large and such that $\displaystyle \lim_{i\to \infty} N(\ga_i)=\infty$, then for any $m\in \bbN$, 
  for $X_D= \Gamma(D)\setminus \bbH^3$, $X_i:=\Gamma(\ga_i)\setminus \bbH^3$ and for $L_m$ the bundle of free $\bbZ$-modules induced by $\varrho_m$ and the lattice $\Lambda_m$, one has that 
\begin{equation}
\liminf_{i\to \infty} \frac{\log |H^2_{\tor}(\bX_i;L_{m})| }{[\Gamma(D): \Gamma(\ga_i)]}\ge  -t^{(2)}_{\bbH^3}(\varrho_m) \vol(X_D)>0.
\label{int.6}\end{equation}
Furthermore, when $m=2\ell$ is even, then  
\begin{equation}
\liminf_{i\to \infty} \frac{\log |H^2_{\tor}(\bX_i;L_{2\ell})| }{[\Gamma(D): \Gamma(\ga_i)]}\ge \left(\frac{\ell(\ell+1)+\frac16}{\pi}\right) \vol(X_D)>0.
\label{int.7}\end{equation}
\label{gt.50}\end{corollary}
\begin{remark}
This is a generalization and a strengthening of \cite[Theorem~1.1]{Pfaff2014b}, since Corollary~\ref{gt.50} does not require $m$ to be even, the congruence subgroups 
$\Gamma(\mathfrak{a}_i)$ need not be included in a common torsion-free subgroup of $\Gamma(D)$, and our constant on the right hand side of \eqref{int.7} is slightly bigger.   
\end{remark}
\begin{remark}
An upper bound of the $\limsup$ was obtained in 
\cite[Theorem A]{Raimbault2013}. The upper bound is again given in terms of
the $L^2$-torsion. However, we point out that the $L^2$-torsion used in 
\cite{Raimbault2013} is defined over $\bbR$ while we define instead the 
$L^2$-torsion over $\bbC$. Hence  the upper bound in \cite{Raimbault2013}  is in fact twice our lower bound, so we cannot use this upper bound to show that the limit exists.  See the discussion at the end of \S~\ref{d3.0} for more details.  
\end{remark}
Theorem~\ref{gt.30} can also be applied to Hecke subgroups of Bianchi groups.  For a non-zero ideal $\mathfrak{a}$ of $\cO_D$, let $\Gamma_0(\mathfrak{a})$ be the associated Hecke subgroup of \eqref{gt.49} below.  Such a group is never torsion-free, but fixing a torsion-free subgroup $\Gamma'$ of $\Gamma(D)$ of finite index satisfying \eqref{int.2} for all $\Gamma'$-cuspidal parabolic subgroups, we can  consider instead  the subgroup $\Gamma_0'(\mathfrak{a}):=\Gamma_0(\mathfrak{a})\cap \Gamma'$.  
\begin{corollary}
If $\ga_i$ is a sequence of non-zero ideals in $\cO_D$ such that $\displaystyle \lim_{i\to \infty} N(\ga_i)=\infty$, then for $m\in \bbN$, $X_D= \Gamma(D)\setminus \bbH^3$, $X_i:=\Gamma'_0(\ga_i)\setminus \bbH^3$ and for $L_m$ the bundle of free $\bbZ$-modules on $X_i$ induced by $\varrho_m$ and the lattice $\Lambda_m$, one has that
\begin{equation}
\liminf_{i\to \infty} \frac{\log |H^2_{\tor}(\bX_i;L_{m})| }{[\Gamma(D): \Gamma_0'(\ga_i)]}\ge -t^{(2)}_{\bbH^3}(\varrho_m) \vol(X_D)>0.
\label{gt.52b}\end{equation}
Again, when $m=2\ell$, this can be rewritten as
$$
\liminf_{i\to \infty} \frac{\log |H^2_{\tor}(\bX_i;L_{2\ell})| }{[\Gamma(D): \Gamma_0'(\ga_i)]}\ge  \left(\frac{\ell(\ell+1)+\frac16}{\pi}\right) \vol(X_D)>0.
$$
\label{gt.52}\end{corollary}
The above results  support the following conjecture.

\begin{conjecture} Let $d=2n+1$ and let $\Gamma\subset \SO_0(d,1)$ be an
arithmetic subgroup. Let $\{\Gamma_i\}$ be a cusp-uniform sequence of torsion 
free congruence subgroups of $\Gamma$ such that
$[\Gamma\colon\Gamma_i]\to\infty$ as $i\to\infty$. Put $X_i=\Gamma_i\backslash
{\mathbb H}^d$. Let $V_\Z$ be an arithmetic $\Gamma$-module
and let $E_\Z$ be the associated local system of free $\Z$-modules over $X_i$. 
Then
\[
\lim_{i\to\infty}\frac{\log|H^j_{\tor}(\bX_i;E_\Z)|}{[\Gamma\colon\Gamma_i]}=
\begin{cases}(-1)^n 2t^{(2)}_{\bbH^d}(\varrho) \vol(\Gamma\backslash\bbH^d),\;& 
j=n+1,\\0,\;& j\neq n+1.
\end{cases}
\]
\end{conjecture}
If $\Gamma$ is cocompact, this is conjecture 1.3 in \cite{BV}.

The most interesting case is the case of trivial coefficients, i.e., 
$H^j_{\tor}(\Gamma_N;\Z)$. Our method does not cover this case. There are several
new problems that arise if the representation $\rho$ does not satisfy 
\eqref{int.3}. In particular, the interior cohomology does not necessarily
 vanish. To control the growth of the covolume of $H^j_{\free}(\Gamma_N;V_\Z)$ 
in this case turns out to be rather difficult. The case of a cocompact 
lattice in $\SL(2,\C)$ and trivial coefficients has been considered in 
\cite{BHV}.

Finally, in dimension 3, our method also allows us to obtain growth of torsion when the group $\Gamma\subset \Gamma(D)$ is fixed and we let $m$ go to infinity.
\begin{theorem}
Let $\Gamma\subset \Gamma(D)$ be a torsion-free finite index subgroup such that \eqref{int.2} holds for each $\Gamma$-cuspidal parabolic subgroup.  Let $X= \Gamma\setminus \SL(2,\bbC)/\SU(2)$ be the corresponding finite volume hyperbolic $3$-manifold.  Then 
$$
      \liminf_{m\to \infty} \frac{\log |H^2_{\tor}(X;L_m)|}{m^2}\ge \frac{\vol(X)}{2\pi} \quad \mbox{and} \quad \limsup_{m\to \infty} \frac{\log |H^2_{\tor}(X;L_m)|}{m^2}\le \frac{\vol(X)}{\pi}.$$
\label{sr.6}\end{theorem}
\begin{remark}
This is a generalization and a sharpening of the exponential growth of torsion obtained by Pfaff and Raimbault in \cite[Theorem~A]{Pfaff-Raimbault}.  Moreover, notice that using Poincaré-Lefschetz duality, our proof avoids dealing directly with Eisenstein series, so does not require estimating the denominator of the $C$-matrix, one of the difficult steps in the proof of \cite{Pfaff-Raimbault}.    Unfortunately, the upper and lower bounds in Theorem~\ref{sr.6} differ by a factor of $2$, so we cannot conclude that the limit exists.  
\label{int.8}\end{remark}

To prove these results, we proceed as in the cocompact case \cite{BV}. For our
approach we have to work with the self-dual bundle $\bE$. Let
$T(X;\bE,g_{X},h_E)$ be the regularized analytic torsion of $X$ and $\bE$ with
respect to the metrics $g_X$ on $X$ and $h_E$ in $E$,
which was introduced in \cite{MP2012}. For a sequence $\{\Gamma_i\}$ of 
subgroups of 
$\Gamma$, which satisfy the assumptions made in Theorem \ref{gt.30}, it follows 
from \cite[Theorem 1.1]{MP2014} that
\begin{equation}
   \lim_{i\to \infty} \frac{\log T(X_i;\bE,g_{X_i},h_E)}{[\Gamma: \Gamma_i]}= t^{(2)}_{\bbH^d}(\overline{\varrho})\vol(X)
   \label{gt.33a}\end{equation}
   where $t^{(2)}_{\bbH^d}(\overline{\varrho})$ is the `local' logarithm of the 
$L^2$-torsion appearing in \cite[(1.1)]{MP2014}. By \cite[Proposition~5.2]{BV}, we have $(-1)^n t^{(2)}_{\bbH^d}(\overline{\varrho})>0$. For $Y\gg 0$ let $X(Y)$ 
be the truncated manifold $X$ at level $Y$. Let $\tau(X(Y);\bE,\bmu_X)$ be the 
Reidemeister torsion of $X(Y)$ and $\bE$ with respect to the basis 
$\bmu_X$ of $H^*(X(Y),\bE)$ given by Eisenstein series \cite{Pfaff2017},
\cite{MR1}. Using the main result of \cite{MR1}, relating $T(X;\bE,g_{X},h_E)$
to $\tau(X(Y);\bE,\bmu_X)$, it follows that
\begin{equation}
   \lim_{i\to \infty} \frac{\log \tau(X_i(Y);\bE,\bmu_{X_i})}{[\Gamma: \Gamma_i]}= 
t^{(2)}_{\bbH^d}(\overline{\varrho})\vol(X).
   \label{gt.33b}\end{equation}
Now recall \cite[(1.4)]{Cheeger1979}, \cite{BV} that the Reidemeister torsion 
is related to the cohomology by
\begin{equation}
\tau(X_i(Y),\bE,\bmu_{X_i})^2= \prod_q \left( \frac{|H^q_{\tor}(X_i(Y);\bE_{\bbZ}|)}{\vol_{\bmu^{\bbR}_{X_i}}\left(H^q_{\free}(X_i(Y);\bE_{\bbZ})\right)} \right)^{(-1)^{q+1}},
\label{gt.2}\end{equation}
where $\bmu^{\bbR}_{X_i}= \{\bmu_{X_i}, \sqrt{-1}\bmu_{X_i}  \}$ is the real basis associated to the complex basis $\bmu_{X_i}$ and 
$\vol_{\bmu^{\bbR}_{X_i}}\left(H^q_{\free}(X_i(Y);\bE_{\bbZ})\right)$ is the covolume of the lattice $H^q_{\free}(X_i(Y);\bE_{\bbZ})$ with respect to the basis 
$\bmu^{\bbR}_{X_i}$. In contrast to the 
cocompact case, the cohomology $H^*(X_i(Y);\bE_\Z)$ never vanishes. It is 
spanned by Eisenstein cohomology classes. To estimate the growth of the 
covolumes of the lattices given by the free part of the cohomology, 
one needs to understand how the Eisenstein cohomology classes 
embed into the boundary cohomology, which is a delicate matter since this 
embedding is rational, but not integral. To overcome this difficulty, our 
strategy, inspired 
by the discussion in \cite[p.680]{Pfaff2017}, is to use Poincar\'e-Lefschetz 
duality \cite{Milnor1962} to replace the Reidemeister torsion 
$\tau(X_i(Y),\bE,\bmu_{X_i})$ by the Reidemeister torsion 
$\tau(X_i(Y),\pa X_i(Y),\bE,\bmu_{X_i,\pa X_i})$ of the relative complex. 
  One then has to deal with a map from boundary cohomology to relative 
cohomology for which covolumes of lattices of free integral cohomology groups 
can be easily estimated, see in particular \S~\ref{mr.0}.  Notice however that, compared to the cocompact case, our approach, through the estimate 
\eqref{gt.2c}, introduces the loss of a factor of $2$ in the exponential 
growth of torsion.

The paper is organized as follows.  In \S~\ref{gt.0}, we recall the results  of \cite{MR1} that we need.  In \S~\ref{co.0}, we derive estimates for the covolumes of lattices of free integral cohomology groups on the boundary of the Borel-Serre compactification.  This is used in \S~\ref{mr.0} to prove the main results.  In \S~\ref{d3.0}, we derive the specific applications of  Theorem~\ref{gt.30} in dimension 3.  Finally, in \S~\ref{sr.0}, we give a proof of Theorem~\ref{sr.6} for sequences of arithmetic $\Gamma$-modules.

\begin{acknowledgements}
The authors are grateful to the hospitality of  the Centre International de Rencontres Math\'ematiques (CIRM) where this project started.  The second author was supported by NSERC and a Canada Research Chair.  
\end{acknowledgements}

\section{Preliminaries} \label{gt.0}

For $d={2n+1}$ an odd natural number, consider either the groups $G:=\SO_o(d,1)$ and $K:=\SO(d)$ or $G:=\Spin(d,1)$ and $K=\Spin(d)$.  In either case, $K$ is a maximal compact subgroup of $G$ and there is a canonical identification $\bbH^d=G/ K$ with the hyperbolic space of dimension $d$.  In fact, if $\mathfrak{g}$ and $\mathfrak{k}$ are the Lie algebras of $G$ and $K$, if $\vartheta$ is the standard Cartan involution with respect to $K$, if $B$ is the Killing form of $\mathfrak{g}$ and if $\mathfrak{g}=\mathfrak{p}\oplus \mathfrak{k}$ is the Cartan decomposition, then the restriction of 
\begin{equation}
  \langle \xi_1, \xi_2\rangle_{\vartheta}:= -\frac{1}{2(d-1)}B(\xi_1,\vartheta(\xi_2)),  \quad \xi_1,\xi_2\in \mathfrak{g},
\label{pr.1}\end{equation}
 to $\mathfrak{p}$ induces a $G$-invariant metric $\bbH^d$ which is precisely the hyperbolic metric.  Let $\Gamma\subset G$ be a discrete torsion-free subgroup such that 
 $$
       X:=\Gamma\setminus \bbH^d= \Gamma\setminus G/K
 $$
is of finite volume with respect to the induced hyperbolic metric $g_X$.  Let $\varrho:G\to \GL(V)$ be a finite dimensional irreducible complex representation  such that 
\begin{equation}
  \varrho\circ \vartheta\ne \varrho.
\label{pr.2}\end{equation}
Recall from \cite[Lemma~4.3]{Muller1993} that the representation $\varrho$ is automatically unimodular.   Restricting $\varrho$ to $\Gamma$, we can define the flat complex vector bundle 
$$
       E:=\bbH^d\times_{\left. \varrho\right|_{\Gamma}} V\quad \mbox{on} \; X.
$$ 
Restricting $\varrho$ to $K$, we can instead define a homogeneous vector bundle 
$$
    \tE:= G\times_{\left. \varrho\right|_{K}}V \quad \mbox{on} \; G.
$$
By \cite[Lemma~3.1]{MM1963}, there is a $K$-invariant Hermitian product $\langle\cdot,\cdot\rangle$ on $V$, unique up to scaling, such that 
$$
      \langle \varrho(\xi)u,v\rangle = \langle u, \varrho(\xi)v\rangle \quad \forall \ \xi\in \mathfrak{p}, \quad \forall \ u,v\in V.  
$$
In particular, since it is $K$-invariant, it induces a natural Hermitian metric $h_{\tE}$ on $\tE$, as well as on the quotient $\Gamma\setminus \tE$ seen as a vector bundle on $X$, where the action of $\Gamma$ on $\tE$ is given by 
$$
      \gamma\cdot [g,v]_{\tE}= [\gamma g,v]_{\tE}
$$
with $[g,v]_{\tE}$ denoting the point of $\tE$ corresponding to the $K$-orbit of $(g,v)\in G\times V$.  By \cite[Proposition~3.1]{MM1963}, there is an explicit isomorphism between $\Gamma\setminus \tE$ and $E$ given by 
$$
    \begin{array}{llcl} \Phi: & \Gamma\setminus \tE &\to & E \\
            & [g,v]_{\Gamma\setminus \tE} & \mapsto & [g,\varrho(g)v]_E,
    
    \end{array}
$$
where $[g,v]_{\Gamma\setminus \tE}$ and $[g,v]_E$ denote the corresponding points in $\Gamma\setminus \tE$ and $E$ after passing to the quotient with respect to the actions of $K$ and $\Gamma$.  Thus, the Hermitian metric $h_{\Gamma\setminus \tE}$ on $\Gamma\setminus \tE$ and this isomorphism induce a natural Hermitian metric $h_E$ on $E$.  Notice however that since $\left. \varrho\right|_{\Gamma}$ is not unitary, the flat connection of $E$ is not compatible with $h_E$.

Let $G=NAK$ be the Iwasawa decomposition of $G$ as in \cite[\S~2]{MP2012} and let $M$ be the centralizer of $A$ in $K$.  Then $P_0:=NAM$ is a parabolic subgroup of $G$.  More generally, for any other parabolic subgroup $P$ of $G$, there exists $k_P\in K$ such that $P=N_PA_PM_P$ with $N_P= k_PN k_P^{-1}$, $A_P=k_PA k_P^{-1}$ and $M_P= k_P M k_P^{-1}$.  Recall that a parabolic subgroup $P$ of $G$ is \textbf{$\Gamma$-cuspidal} if $\Gamma\cap N_P$ is a lattice in $N_P$.  In this case, we denote this lattice  by 
$$
     \Gamma_P:=\Gamma\cap N_P.
$$ 
Restricting the representation $\varrho$ to $\Gamma_P$, we can define a flat vector bundle 
$$
     E_P= N_P\times_{\left.\varrho\right|_{\Gamma_P}} V
$$
on the quotient $T_P:= \Gamma_P\setminus N_P$.  Since $N_P$ is commutative, notice that the restriction of \eqref{pr.1} to $\mathfrak{p}$ induces a flat Riemannian metric $g_{T_P}$ on $T_P$.  On the other hand, considering the trivial vector bundle  $\tE_P= N_P\times V$ and its quotient $\Gamma_P\setminus \tE_P$ with respect to the action of $\Gamma_P$ on the base $N_P$, there is again an explicit isomorphism 
$$
    \begin{array}{llcl} \Phi_P: & \Gamma_P\setminus \tE_P &\to & E_P \\
            & [n,v]_{\Gamma_P\setminus \tE_P} & \mapsto & [n,\varrho(n)v]_{E_P},
    \end{array}
$$
so that the Hermitian product on $V$ induces a natural Hermitian metric $h_{E_P}$ on $E_P$.  

Let $\mathfrak{P}_{\Gamma}$ be a fixed set of representatives of the $\Gamma$-conjugacy classes of $\Gamma$-cuspidal parabolic subgroups.  Then $\mathfrak{P}_{\Gamma}$ is a finite set with cardinality corresponding to the number of cusps of $X$.  Suppose now that $\Gamma$ is such that for any $\Gamma$-cuspidal parabolic subgroup $P$,
\begin{equation}
  \Gamma\cap P= \Gamma\cap N_P= \Gamma_P.
\label{dc.3}\end{equation}
Let $\mathfrak{a}$ be the Lie algebra of $A$ and equip it with the norm induced by the restriction of \eqref{pr.1} to $\mathfrak{p}$.  Let $H_1\in \mathfrak{a}$ be the unique element of norm $1$ such that the positive restricted root involved in the choice of $N$ is positive on $H_1$.  Consider then the unique group isomorphism 
$$
      R: \bbR^+\to A
$$ 
whose differential at $t=1$ sends $1$ onto $H_1$.  More generally, denote by $R_P: \bbR^+\to A_P$ the group isomorphism given by $R_P(t):= k_P R(t) k_P^{-1}$ and set
$$
       A_P^0[Y];= R_P((Y,\infty))  \quad \mbox{for} \; Y>0.
$$ 
Then there exists $Y_0>0$ such that for $Y\ge Y_0$, there is a compact subset $C(Y)$ of $G$ and a decomposition 
$$
    G= \Gamma\cdot C(Y) \sqcup \bigsqcup_{P\in \mathfrak{P}_{\Gamma}} \Gamma\cdot N_P A^0_P[Y]K
$$
such that for each $P\in \mathfrak{P}_{\Gamma}$, 
$$
     \gamma\cdot N_PA^0_P[Y]K \cap N_PA^0_P[Y]K\ne \emptyset \; \Longleftrightarrow \; \gamma\in \Gamma_P.
$$
On $X$, this induces the decomposition 
$$
       X= X(Y)\sqcup \bigsqcup_{P\in \mathfrak{P}_{\Gamma}} F_P(Y)
$$
where 
$$
  F_P(Y):= \Gamma_P\setminus N_P\times A^0_P[Y]= T_P\times (Y,\infty)
$$
with the hyperbolic metric $g_X$ on $F_P(Y)$ given by 
\begin{equation}
     \frac{dt^2+ g_{T_P}}{t^2}, \quad t\in (Y,\infty).
\label{hm.1}\end{equation}
Moreover, there is an explicit isomorphism 
$$
      \begin{array}{llcl} \psi_P: & \left. E\right|_{F_P(Y)} &\to & \pr_1^*E_P \\
            & [nR_P(r),\varrho(nR_{P}(r))v]_{E} & \mapsto & [(n,r),\varrho(nR_P(r))v]_{\pr_1^*E_P},
    \end{array}
 $$
where $\pr_1: T_P\times A^0_P[Y]\to T_P$ is the projection on the first factor.  Thus, in the cusp end $F_P(Y)$, we can work directly with $\pr_1^*E_P$ instead of $E$, keeping in mind that the Hermitian metric is not quite the one of $E_P$ pull-back to $F_P(Y)$, but the one given by
$$
       h_E([(n,r),\varrho(nR_P(r))u]_{\pr_1^*E_P}, [(n,r),\varrho(nR_P(r))v]_{\pr_1^*E_P})= \langle u,v\rangle
 $$
in terms of the Hermitian product $\langle\cdot,\cdot\rangle$ on $V$.  Now, on $T_P$, Hodge theory  induces an identification between the cohomology group $H^q(T_P;E_P)$ and the space of harmonic forms $\cH^q(T_P;E_P)$ with respect to the Hodge Laplacian associated to the metric $g_{T_P}$ and the Hermitian metric $h_{E_P}$.  If $\mathfrak{n}_P$ denotes the Lie algebra of $N_P$, there is also a natural inclusion
\begin{equation}
     \begin{array}{lccc} \iota_P:  & \Lambda^q\mathfrak{n}_P^*\otimes V & \to & \Omega^q(T_P;E_P) \\
            & \omega\otimes v & \mapsto & \widehat{\omega}\otimes \widehat{v},
    \end{array}\label{dc.16}\end{equation}
where $\widehat{v}(\Gamma_Pn):= [n,\varrho(n)v]_{E_P}$ and $\left.\widehat{\omega}\right|_{\Gamma_P n}=\omega$ via the natural identification
$$
           \Lambda^q(T^*(T_P))= T_P\times \Lambda^q\mathfrak{n}^*_P.
$$
By van Est's theorem \cite{vE} (see also \cite[Lemma~2.6]{MR1}), the map \eqref{dc.16} induces an isomorphism between $H^q(T_P;E_P)$ and the cohomology group $H^q(\mathfrak{n}_P;V)$ of degree $q$ of the $V$-valued Lie algebra complex 
$\Lambda^*\mathfrak{n}_P^*\otimes V$ with differential
$$
      d_{\mathfrak{n}_P} \Phi(T_1,\ldots, T_{q+1})= \sum_{i=1}^{q+1} (-1)^{i+1}\varrho(T_i)\Phi(T_1,\ldots, \widehat{T}_i,\ldots, T_{q+1}),  \quad \Phi\in \Lambda^q\mathfrak{n}_P^*,
$$
where ``$\widehat{\; \; }$" above a variable denotes omission.  Using  the inner product \eqref{pr.1} and the Hermitian product $\langle\cdot,\cdot\rangle$ induces a natural Hermitian product on $\Lambda^q\mathfrak{n}^*_P\otimes V$. Hence, as in \cite{Kostant}, we can consider the adjoint $d^*_{\mathfrak{n}_P}: \Lambda^*\mathfrak{n}_P^*\otimes V\to 
\Lambda^{*-1}\mathfrak{n}_P^*\otimes V$ of $d_{\mathfrak{n}_P}$ and the corresponding de Rham and Hodge operators
$$
     K_P:= d_{\mathfrak{n}_P}+ d_{\mathfrak{n}_P}^*, \quad L_P:=d_{\mathfrak{n}_P}d^*_{\mathfrak{n}_P} + d^*_{\mathfrak{n}_P}d_{\mathfrak{n}_P}.
$$
Hodge theory in this setting identifies the Lie algebra cohomology $H^q(\mathfrak{n}_P;V)$ with the kernel of $L_P$,
$$
    \cH^q(\mathfrak{n}_P;V):=\{ \Phi\in \Lambda^q\mathfrak{n}_P^*\otimes V \; | \; L_P\Phi=0\}.
$$
 Thus, to summarize, we have the natural identifications
\begin{equation}
  H^q(T_P;E_P)\cong \cH^q(T_P;E_P)\cong H^q(\mathfrak{n}_P;V)\cong\cH^q(\mathfrak{n}_P;V).
\label{pr.4}\end{equation}
The first identification suggests to take a basis $\mu_{T_P}^q$ of $H^q(T_P;E_P)$ given by orthonormal harmonic forms.  Since 
$$
     H^q(\pa X(Y);E)\cong \bigoplus_{P\in \mathfrak{P}_{\Gamma}} H^q(T_P;E_P),
$$
this induces a corresponding basis 
$$ 
 \mu^q_{\pa X}:= \bigoplus_{P\in\mathfrak{P}_{\Gamma}} \mu^q_{T_P} 
$$
for $H^q(\pa X(Y);E).$   Now, we know from \cite{Pfaff2017} that $H^q(X(Y);E)=\{0\}$ for 
$q<n$ and $q=2n+1$, while otherwise the natural map $\iota_Y: \pa X(Y)\to X(Y)$ induces an inclusion in cohomology
$$
     \iota_Y^* : H^q(X(Y);E) \hookrightarrow H^q(\pa X(Y);E)
$$
which is an isomorphism for $n<q\le 2n$ and a strict inclusion for $q=n$.  Furthermore, in this latter case, as explained in \cite{Pfaff2017} and \cite{MR1}, there is an orthonormal decomposition 
\begin{equation}
  H^n(T_P;E_P)= H^n_+(T_P;E_P)\oplus H^n_-(T_P;E_P)
\label{pr.5}\end{equation}
and a corresponding orthonormal decomposition
\begin{equation}
  H^n(\pa X(Y);E)= H^n_+(\pa X(Y);E)\oplus H^n_-(\pa X(Y);E).
\label{pr.5a}\end{equation}
If $\pr_-: H^n(\pa X(Y);E)\to H^n_-(\pa X(Y);E)$ is the projection on the second factor, then we know from \cite{Pfaff2017} or \cite{MR1} that 
\begin{equation}
     \pr_-\circ \iota_Y^*: H^n(X(Y);E)\to H^n_-(\pa X(Y);E)
\label{pr.5b}\end{equation}
is an isomorphism.  Thus, assuming without loss of generality that $\mu^n_{\pa X}= (\mu^n_{+}, \mu^n_{-})$ with $\mu^n_{\pm}$ an orthonormal basis of $H^n_{\pm}(\pa X(Y);E)$,  we see that a natural choice of basis $\mu^q_X$ for $H^q(X(Y);E)$ is obtained by requiring that  
\begin{equation}
\iota^*_Y \mu^q_X= \mu^q_{\pa X} \quad  \mbox{for}\; n<q\le 2n \quad  \mbox{and}  \quad \pr_-\circ \iota^*_Y\mu_X^n= \mu_-^n \quad  \mbox{for} \; q=n.
\label{bc.1}\end{equation}  With this choice of basis, the following relation is obtained between the Reidemeister torsion $\tau(X(Y),E,\mu_X)$ of the flat vector bundle $E$ and the analytic torsion $T(X;E,g_X,h_E)$.

\begin{theorem}[\cite{MR1}]
Let $\Gamma$ be a torsion-free subgroups of $G$ such that \eqref{dc.3} holds and  $X=\Gamma\setminus G/ K$ is a finite volume hyperbolic manifolds.  Let $\varrho: G\to \GL(V)$ be a finite dimension irreducible complex representation satisfying \eqref{pr.2} and let $E$ be the associated flat vector bundle on $X$ with Hermitian metric $h_E$.  In this case,
$$
      \log T(X;E,g_X,h_E)= \log \tau(X(Y),E,\mu_X)- \frac12 \tau(\pa X(Y),E,\mu_{\pa X}) + \kappa(X) c_{\varrho},
$$
where $\kappa(X)=\# \mathfrak{P}_{\Gamma}$ is the number of cusps and $c_{\varrho}$ is a constant depending only on $\varrho$.  
\label{pr.6}\end{theorem}
\begin{remark}
The result of \cite{MR1} is slightly more general, since the assumption that \eqref{dc.3} holds can be relaxed when $G=\Spin(d,1)$, see \cite[Assumption~2.2]{MR1}. 
\end{remark}

\section{Estimates of covolumes}\label{co.0}

Let $\G$ be a quasi-split semi-simple algebraic group defined over $\Q$ such 
that $G:=\G(\R)$ is isomorphic to $\SO(d,1)$ or $\Spin(d,1)$ with $d=2n+1$. We
assume that $\G$ is not anisotropic. 
Let $\Gamma\subset \G(\Q)$ be an arithmetic subgroup. Then $\Gamma\subset 
G(\R)$ is not cocompact. Let ${\boldsymbol \varrho}$ be a rational representation of $\G$
on a $\Q$-vector space $V_\Q$. Let $V=V_\Q\otimes_\Q\C$. Let 
$\varrho\colon G\to \GL(V)$ be the restriction to $G$ 
of the representation ${\boldsymbol \varrho}_\C$ of $\G(\C)$ on $V$, induced by ${\boldsymbol \varrho}$. 
Then $\varrho$
is the direct sum of finitely many irreducible representations $\varrho_i$. We
we assume that each $\varrho_i$ satisfies \eqref{int.3}. The existence of such
rational representations is proved in \cite[sect. 8.1]{BV}. 

For the
convenience of the reader we recall the construction. Let $\T\subset\G$ be a
maximal torus. Let $F$ be a Galois extension of $\Q$ over which $\G$ splits. 
Let $X^\ast$ resp. $X_+^\ast$ be the character lattice of $\T_F=\T\otimes_\Q F$,
resp. the dominant characters of $\T_F$. Given $x\in X^\ast_+$, let $\varrho_x$
be the unique irreducible representation of $\G\times_\Q F$ with highest 
weight $x$ on a $F$-vector space $W_x$. Let $V_x:=\Res_{F/\Q}(W_x)$,
which means $W_x$ considered as a $\Q$-vector space. Let $\tilde\varrho_x\colon
\G\to\GL(V_x)$ be the $\Q$-rational representation, which is the 
composition of the representations $\G\to\Res_{F/\Q}\GL(W_x)$ and 
$\Res_{F/\Q}(W_x)\hookrightarrow \GL(\Res_{F/\Q}W_x)$. The highest weights of the
irreducible components $\varrho_i$ of the induced 
representation  of $\G(\C)$ on $V_x\otimes_\Q\C$ 
are obtained from $x$ by applying the various embeddings of $F$ into $\C$. 
In the lemma in section 8.1 of \cite{BV} it is proved that if $x$ is in the 
complement of the union of finitely many hyperplanes, then every 
$\varrho_i$ satisfies condition \eqref{int.3}. 

Let $V_\R$ denote $V$ seen as a real vector space. Since $\Gamma$ is 
arithmetic, there exists a $\Gamma$-invariant lattice $V_\Z$ in $V_\R$. 
We call the $\Gamma$-module $V_\Z$ which satisfies the above conditions 
 a {\it strongly $L^2$-acyclic arithmetic $\Gamma$-module}. 

Now, let $\{\Gamma_i\}_{i\in \bbN_0}$ be a  sequence of finite-index subgroups of $\Gamma$
\begin{equation}
 \Gamma_i\subset \Gamma_0:=\Gamma.  
\label{gt.7}\end{equation}   
Then for each $i$, $X_i:= \Gamma_i\setminus G/ K$ is a covering of $X= \Gamma\setminus G/K$.  We note that $X$ and $X_i$ are of  finite volume.  For $i>0$, we will also assume that $\Gamma_i$ is torsion-free , so that $X_i$ is a smooth hyperbolic manifold of finite volume, and that \eqref{dc.3} holds for $\Gamma_i$. However, we will not require that $\Gamma$ is torsion-free or satisfies \eqref{dc.3}, so that  $X_0$ may have  orbifold singularities.
Since $\Gamma$ preserves the $\bbZ$-module $V_{\bbZ}$, notice that $E=E_{\bbZ}\otimes_{\bbZ}\bbR$ for a natural smooth bundle of free $\bbZ$-modules 
\begin{equation}
   E_{\bbZ}:= \bbH^d\times_{\left.\varrho\right|_{\Gamma}} V_{\bbZ} \quad \mbox{on} \; X.
\label{gt.9}\end{equation}
To lighten the notation, we will also denote by $E= \bbH^d\times_{\left.\varrho\right|_{\Gamma_i}}V$ and $E_{\bbZ}=\bbH^d\times_{\left.\varrho\right|_{\Gamma_i}} V_{\bbZ}$ the corresponding bundles on $X_i$, trusting this shall lead to no confusion.  Alternatively, notice that these bundles can also be obtained by pulling back $E$ and $E_{\bbZ}$ under the covering map $X_i\to X$.  The cohomology group of cochains of degree $q$ with local coefficients in $E_{\bbZ}$ are then $\bbZ$-modules  admitting a decomposition into a free part
$H^q_{\free}(X_i;E_{\bbZ})$ and a torsion part $H^q_{\tor}(X_i;E_{\bbZ})$,
\begin{equation}
  H^q(X_i;E_{\bbZ})= H^q_{\free}(X_i;E_{\bbZ}) \oplus H^q_{\tor}(X_i;E_{\bbZ}),
  \label{gt.10}\end{equation} 
so that in particular $H^q(X_i;E)= H^q(X_i ;E_{\bbZ})\otimes_{\bbZ}\bbR= H^q_{\free}(X_i;E_{\bbZ})\otimes_{\bbZ}\bbR$.  If $H_q(X_i;E_{\bbZ})$ denotes the corresponding homology group of chains of degree $q$ with local coefficients in $E_{\bbZ}$, then we have a similar decomposition in terms of free and torsion parts,
\begin{equation}
  H_q(X_i;E_{\bbZ})= H_q(X_i;E_{\bbZ})_{\free} \oplus H_q(X_i;E_{\bbZ})_{\tor}.
  \label{gt.10a}\end{equation}
Recall from the universal coefficient theorem that there are natural isomorphisms
\begin{equation}
\begin{gathered}
H^q_{\free}(X_i;E_{\bbZ})\cong H_q(X_i;E_{\bbZ}^*)_{\free}, \\
H^q_{\tor}(X_i;E_{\bbZ})\cong H_{q-1}(X_i;E_{\bbZ}^*)_{\tor},
\end{gathered}
\label{gt.11}\end{equation}
where $E^*_{\bbZ}$ is the bundle of free $\bbZ$-modules dual to $E_{\bbZ}$.  To deal with self-dual bundles of free $\bbZ$-modules, we will consider the bundle of free $\bbZ$-modules 
\begin{equation}
 \bE_{\bbZ}:= E_{\bbZ}\oplus E_{\bbZ}^{*}
\label{gt.12}\end{equation}
with corresponding self-dual flat vector bundle $\bE:= E\oplus E^*$.

  For each $i\in \bbN_0$ and $P$ a $\Gamma$-cuspidal parabolic subgroup, set
$$
        \Gamma_{P,i}:= \Gamma_i\cap N_P\quad \mbox{and}  \quad T_{P,i}= \Gamma_{P,i}\setminus N_P
$$
with flat metric $g_{T_{P,i}}$ as in \eqref{hm.1}.  Recall that on $T_{P,0}=T_P$, we have the natural flat vector bundle $E_P= N_P\times_{\left.\varrho\right|_{\Gamma_P}}V$.  We can in a similar way define a bundle of free $\bbZ$-modules $E_{P,\bbZ}= N_P\times_{\left.\varrho\right|_{\Gamma_P}}V_{\bbZ}$.  To lighten the notation, for each $i\in \bbN$, we will also denote by $E_P:=N_P\times_{\left.\varrho\right|_{\Gamma_{P,i}}}V$ and $E_{P,\bbZ}= N_P\times_{\left.\varrho\right|_{\Gamma_{P,i}}}V_{\bbZ}$ the corresponding bundles over $T_{P,i}$.  Alternatively, notice that these bundles can be obtained by pulling back $E_P$ and $E_{P,\bbZ}$ under the natural covering map $\pi_i: T_{P,i}\to T_P$.  As in \eqref{gt.12}, we consider the corresponding self-dual bundles
\begin{equation}
   \bE_P:= E_P\oplus E_P^*, \quad \mbox{and}  \quad  \bE_{P,\bbZ}:= E_{P,\bbZ}\oplus E^*_{P,\bbZ}.  
\label{gt.16}\end{equation}  
As in \eqref{pr.4},  Hodge theory and the van Est's isomorphism induced by the map \eqref{dc.16} yield natural identifications
\begin{equation}
  H^q(T_{P,i}; \bE_P)\cong \cH^q(T_{P,i};\bE_P)\cong H^q(\mathfrak{n}_P;\bV)\cong \cH^q(\mathfrak{n}_P;\bV),  
\label{gt.17}\end{equation}
where $\bV=V\oplus V^*$ and $\cH^q(T_{P,i};\bE_P)$ is the space of harmonic forms of degree $q$.  In particular, notice that the dimension of $H^q(T_{P,i}; \bE_P)$ does not depend on $i\in \bbN_0$.  More importantly, via the identifications \eqref{gt.17}, the natural metric on  $\cH^q(T_{P,i};\bE_P)$ induces one on $H^q(T_{P,i}; \bE_P)$.  The space $\cH^q(\mathfrak{n}_P;\bV)$ also comes with a natural metric, but under the identification \eqref{gt.17}, it does not quite give the same metric as the one on $\cH^q(T_{P,i};\bE_P)$.  
Indeed, if $\langle\cdot, \cdot\rangle_{T_{P,i}}$ and $\langle\cdot, \cdot\rangle_{\mathfrak{n}_P}$ are the inner products on $\cH^q(T_{P,i};\bE_P)$ and $ \cH^q(\mathfrak{n}_P;\bV)$ respectively, then under the identification \eqref{gt.17}, we have that
\begin{equation}
   \langle\cdot, \cdot\rangle_{T_{P,i}}= \vol(T_{P,i}) \langle\cdot, \cdot\rangle_{\mathfrak{n}_P},
\label{gt.18}\end{equation}
where $\vol(T_{P,i})$ is the volume of $T_{P,i}$ with respect to the metric $g_{T_{P,i}}$.  Let us denote by $\vol(H^q_{\free}(T_{P,i};\bE_{P,\bbZ}))$ the covolume of the lattice 
$H^q_{\free}(T_{P,i};\bE_{P,\bbZ})$ in $H^q(T_{P,i};\bE_P)$ with respect to the inner product $\langle\cdot, \cdot\rangle_{T_{P,i}}$.
\begin{proposition}
For each $i\in \bbN_0$, 
$$
         \frac{1}{\vol(H^{n-q}_{\free}(T_{P};\bE_{P,\bbZ})) [\Gamma_{P}: \Gamma_{P,i}]^{\frac{b_{q}}2}}\le \vol(H^q_{\free}(T_{P,i};\bE_{P,\bbZ}))\le  \vol(H^q_{\free}(T_{P};\bE_{P,\bbZ}))[\Gamma_P: \Gamma_{P,i}]^{\frac{b_q}{2}},
$$
where $b_q:=\dim_{\bbR}\cH^q(\mathfrak{n}_P;\bV)$.
\label{gt.19}\end{proposition}
\begin{proof}
Let $\pi_i: T_{P,i}\to T_P$ be the natural covering map.  Then $\pi_i^*(H^q_{\free}(T_P;\bE_{P,\bbZ}))$ is a sublattice of $H^q_{\free}(T_{P,i};\bE_{P,\bbZ})$.  Hence, we see that
\begin{equation}
\begin{aligned}
\vol(H^q_{\free}(T_{P,i};\bE_{P,\bbZ}))& \le \vol(\pi_i^* (H^q_{\free}(T_P;\bE_{P,\bbZ}))) = \left( \frac{\vol(T_{P,i})}{\vol(T_P)} \right)^{\frac{b_q}2}\vol(H_{\free}^q(T_{P};\bE_{P,\bbZ})), \\
                                                         &= [\Gamma_P: \Gamma_{P,i}]^{\frac{b_q}2}\vol(H_{\free}^q(T_{P};\bE_{P,\bbZ})),                                                      
\end{aligned}
\label{gt.20}\end{equation}
giving the inequality on the right.  For the inequality on the left, notice that Poincar\'e duality induces an isomorphism $H^{n-q}_{\free}(T_{P,i};\bE_{P,\bbZ})^*\cong H^q_{\free}(T_P;\bE_{P,\bbZ})$ and an isometry 
$$
H^{n-q}(T_{P,i};\bE_{P})^*\cong H^q(T_P;\bE_{P}),
$$
 so that $\vol(H^{n-q}_{\free}(T_{P,i};\bE_{P,\bbZ})^*)=\vol(H^q_{\free}(T_P;\bE_{P,\bbZ}))$.  Since essentially by definition
 $$
 \vol(H^{n-q}_{\free}(T_{P,i};\bE_{P,\bbZ})^*)=\vol(H^{n-q}_{\free}(T_{P,i};\bE_{P,\bbZ}))^{-1},
 $$ 
 we  see from \eqref{gt.20} that
$$
\vol(H^q_{\free}(T_{P,i};\bE_{P,\bbZ}))= \frac{1}{\vol(H^{n-q}_{\free}(T_{P,i};\bE_{P,\bbZ}))}\ge \frac{1}{[\Gamma_P: \Gamma_{P,i}]^{\frac{b_{n-q}}2}\vol(H_{\free}^{n-q}(T_{P};\bE_{P,\bbZ}))}.
$$
Since $b_{n-q}=b_q$ by Poincar\'e duality, this gives the left inequality.
\end{proof}

For $q=n$, recall that we have an orthogonal decomposition 
$$
      \cH^n(\mathfrak{n}_P,\bV)= \cH^n_+(\mathfrak{n}_P,\bV)\oplus \cH^n_-(\mathfrak{n}_P,\bV),
$$
which via the identifications \eqref{gt.17} induces a corresponding orthogonal decomposition
\begin{equation}
  H^n(T_{P,i};\bE_P)= H^n_+(T_{P,i};\bE_P)\oplus H^n_-(T_{P,i};\bE_P).
\label{gt.26}\end{equation}
\begin{lemma}
The restriction of $H^n_{\free}(T_{P,i};\bE_{P,\bbZ})$ to $H^n_{\pm}(T_{P,i};\bE_P)$ induces a lattice $L_{i,\pm}$ in $H^n_{\pm}(T_{P,i};\bE_P)$.
\label{gt.27}\end{lemma}
\begin{proof}
The decomposition of $H^n(\mathfrak{n}_P;\bV)$ is in terms of the eigenspaces of the action of $(\Ad^*\otimes\varrho)(R_P(t))$,
$$
     (\Ad^*\otimes\varrho)(R_P(t)) \Phi_{\pm}= t^{\lambda_{\varrho,n}^{\pm}-n}\Phi_{\pm} \quad \mbox{for} \quad \Phi_{\pm}\in \cH^n_{\pm}(\mathfrak{n}_P;V).
$$
But by \cite[(6.4), (6.6)]{Pfaff2017}, we know that the constants $\lambda_{\varrho,n}^{\pm}$ are integers or half-integers.  In particular, the eigenvalues of $(\Ad^*\otimes\varrho)(R_P( 4))$ are positive integers.  We know also that $(\Ad^*\otimes \varrho)(R_P(4))$ has rational coefficients.  Hence, the projections 
$$
    \pr_{\pm}: \cH^n(\mathfrak{n}_P;\bV)\to \cH^n_{\pm}(\mathfrak{n}_P;\bV)
$$
 are given by matrices with rational coefficients.  Using the natural identification of \eqref{gt.17} induced by the map \eqref{dc.16}, we see that the corresponding projections 
 $$
    \pr_{\pm}: H^n(T_{P,i};\bE_{P})  \to H^n_{\pm}(T_{P,i};\bE_{P}) 
 $$
 have rational coefficients, from which the result follows.  
\end{proof}
The covolumes $\vol(L_{i,\pm})$ of the lattices $L_{i,\pm}$ in $\cH^n_{\pm}(T_{P,i};\bE_{P})$  can be estimated as follows.

\begin{proposition}
For each $i\in \bbN_0$, 
$$
  \frac{1}{\vol(L_{0,\mp})[\Gamma_P:\Gamma_{P,i}]^{\frac{b_n}4}}\le \vol(L_{i,\pm})\le \vol(L_{0,\pm})[\Gamma_P: \Gamma_{P,i}]^{\frac{b_n}4}.
$$
\label{gt.28}\end{proposition}
\begin{proof}
For the inequality on the right, notice that the decomposition \eqref{gt.26} behaves well under pull-back, so that $\pi_i^*(L_{0,\pm})$ is a sublattice of $L_{i,\pm}$.  Hence,
\begin{equation}
\vol(L_{i,\pm}) \le \vol(\pi^*_i L_{0,\pm})= \left( \frac{vol(T_{P,i})}{\vol(T_P)}\right)^{\frac{b_n}4}\vol(L_{0,\pm})= [\Gamma_P:\Gamma_{P,i}]^{\frac{b_n}4}\vol(L_{0,\pm}),
\label{gt.29}\end{equation}
giving the desired inequality.  For the inequality on the left, notice that since Poincar\'e duality induces an isomorphism $(H^n_{\free}(T_{P,i};\bE_{P,\bbZ}))^*\cong H^n_{\free}(T_{P,i};\bE_{P,\bbZ})$ and an  isometry 
$$(H^n(T_{P,i};\bE_P))^*\cong H^n(T_{P,i};\bE_P),
$$
 we have that
$$
  \vol(H^n_{\free}(T_{P,i};\bE_{P,\bbZ}))= \vol((H^n_{\free}(T_{P,i};\bE_{P,\bbZ}))^*).
$$ 
Since on the other hand we have trivially that 
$$
\vol((H^n_{\free}(T_{P,i};\bE_{P,\bbZ}))^*)= \left(\vol(H^n_{\free}(T_{P,i};\bE_{P,\bbZ}))\right)^{-1},
$$ 
this means that
$$
      \vol(H^n_{\free}(T_{P,i};\bE_{P,\bbZ}))=1.   
$$
Now, $L_{i,+}+ L_{i,-}$ is a sublattice of $H^n_{\free}(T_{P,i};\bE_{P,\bbZ})$.  Hence,
\begin{equation}
   1=\vol(H^n_{\free}(T_{P,i};\bE_{P,\bbZ})) \le \vol(L_{i,+}+ L_{i,-})= \vol(L_{i,+})\vol(L_{i,-}),
\label{lb.45b}\end{equation}
where in the last step we have used that \eqref{gt.26} is an orthogonal decomposition.  Hence, combining this estimate with \eqref{gt.29} gives
the desired inequality
$$
    \vol(L_{i,\pm})\ge \frac{1}{\vol(L_{i,\mp})}\ge \frac{1}{\vol(L_{0,\mp})[\Gamma_P:\Gamma_{P,i}]^{\frac{b_n}4}}.
$$
\end{proof}

\section{Proof of the main result} \label{mr.0}

  Let $\bmu_{X_i}$ be a complex basis of $H^*(X_i(Y); \bE)$ obtained as in \eqref{bc.1}. Using Poincar\'e-Lefschetz duality \cite{Milnor1962} we have that 
$$
     \tau(X(Y),\bE,\bmu_{X_i})= \tau(X_i(Y),\pa X_i(Y),\bE,\bmu_{X_i,\pa X_i}),
$$
where $\bmu_{X_i,\pa X_i}$ is the basis of $H^*(X_i(Y),\pa X_i(Y);\bE)$ dual to $\bmu_{X_i}$.  The basis $\mu_{X_i,\pa X_i}$ has a much simpler description compared to $\bmu_{X_i}$, in particular it does not involve Eisenstein series.  Indeed, consider the long exact sequence 
\begin{equation}
\xymatrix{
     \cdots \ar[r]^-{\pa}  & H^q(X_i(Y),\pa X_i(Y);\bE)\ar[r] & H^q(X_i(Y);\bE) \ar[r] & H^q(\pa X_i(Y);\bE)\ar[r]^-{\pa} &    \cdots
 }   
\label{gt.1}\end{equation}
associated to the pair $(X_i(Y), \pa X_i(Y))$.
\begin{lemma}
In terms of the boundary homomorphism $\pa$ of the long exact sequence \eqref{gt.1},  $\bmu_{X_i,\pa X_i}= \pa ( \bmu_{\pa X_i})$, where $\bmu_{\pa X_i}$ is an orthonormal basis of $H^q(Z_i;\bE)\cong \bigoplus_{P\in\mathfrak{P}_{\Gamma_i}}\cH^q(T_P;\bE_P)$ (with respect to the metrics $g_{T_{P}}$ and $h_{E_P}$, where $Z_i= \pa X_i(Y)$)  for $0\le q<n$ and of 
$$\cH^q_+(Z_i;\bE)\cong\bigoplus_{P\in\mathfrak{P_{\Gamma_i}}} \cH^q_+(T_P;\bE_P)\subset H^q(Z_i;\bE)
$$ 
for $q=n$.
\label{pld.1}\end{lemma}
\begin{proof}
The only delicate point is for $q=n$, in which case it suffices to use the isomorphism \eqref{pr.5b} and to notice that Poincar\'e duality induces an identification 
$$
       (H^n_{\pm}(T_P;\bE_P))^*\cong H^n_{\mp}(T_P;\bE_P).
$$
\end{proof} 
 
  Now, by \cite[(1.4)]{Cheeger1979}, we know that
\begin{equation}
\tau(X_i(Y),\pa X_i(Y),\bE,\bmu_{X_i,\pa X_i})^2= \prod_q \left( \frac{|H^q_{\tor}(X_i(Y), \pa X_i(Y);\bE_{\bbZ}|}{\vol_{\bmu^{\bbR}_{X_i,\pa X_i}}\left(H^q_{\free}(X_i(Y),\pa X_i(Y);\bE_{\bbZ})\right)} \right)^{(-1)^{q+1}},
\label{gt.2a}\end{equation}
where $\bmu^{\bbR}_{X_i,\pa X_i}= \{\bmu_{X_i,\pa X_i}, \sqrt{-1}\bmu_{X_i,\pa X_i}  \}$ is the real basis associated to the complex basis $\bmu_{X_i,\pa X_i}$ and $\vol_{\bmu^{\bbR}_{X_i,\pa X_i}}\left(H^q_{\free}(X_i(Y),\pa X_i(Y);\bE_{\bbZ})\right)$ is the covolume of the lattice 
$$
H^q_{\free}(X_i(Y),\pa X_i(Y);\bE_{\bbZ})
$$ 
with respect to the basis 
$\bmu^{\bbR}_{X_i,\pa X_i}$.  However, we have that 
\begin{equation}
     \vol_{\bmu^{\bbR}_{X_i,\pa X_i}}\left(H^q_{\free}(X_i(Y),\pa X_i(Y);\bE_{\bbZ})\right)= \frac{\vol_{\bmu^{\bbR}_{Z_i}}( H^{q-1}_{\free}(Z_i;\bE_{\bbZ}))}{[H^q_{\free}(X_i(Y),\pa X_i(Y);\bE_{\bbZ}): \pa\left( H^{q-1}_{\free}(Z_i;\bE_{\bbZ})\right)]}
\label{gt.2b}\end{equation}
with $H^{q-1}_{\free}(Z_i;\bE_{\bbZ})$ replaced with $H^{q-1}_{+}(Z_i;\bE)\cap H^{q-1}_{\free}(Z_i;\bE_{\bbZ})$  when $q=n+1$.   Moreover, we see from the lattice version of \eqref{gt.1} that 
\begin{equation}
\begin{aligned}
    1\le [H^q_{\free}(X_i(Y),\pa X_i(Y);\bE_{\bbZ}): \pa\left( H^{q-1}_{\free}(Z_i;\bE_{\bbZ})\right)] & \le |H^q_{\tor}(X_i(Y),\bE_{\bbZ})| \\
     &= |H^{2n+2-q}_{\tor}(X_i(Y),\pa X_i(Y);\bE_{\bbZ}|,
\end{aligned}    
\label{gt.2c}\end{equation}
where on the second line we have used the fact that 
\begin{equation}
H^q_{\tor}(X_i(Y), \bE_{\bbZ})\cong H_{q-1}(X_i(Y);\bE_{\bbZ})_{\tor}\cong H^{2n+2-q}_{\tor}(X_i(Y),\pa X_i(Y), \bE_{\bbZ}) 
\label{gt.3}\end{equation}
by the universal coefficient theorem and Poincar\'e-Lefschetz duality.  Hence, we see that 
\begin{equation}
\begin{aligned}
\frac{\vol_{\bmu_{Z_i}}( H^{q-1}_{\free}(Z_i;\bE_{\bbZ}))}{|H^{2n+2-q}_{\tor}(X_i(Y),\pa X_i(Y);\bE_{\bbZ})|}  & \le \vol_{\bmu^{\bbR}_{X_i,\pa X_i}}\left(H^q_{\free}(X_i(Y),\pa X_i(Y);\bE_{\bbZ})\right)   \\
& \le \vol_{\bmu_{Z_i}}( H^{q-1}_{\free}(Z_i;\bE_{\bbZ})).
\end{aligned}
\label{gt.4}\end{equation}
Using these inequalities in \eqref{gt.2}, as well as \eqref{gt.3} and the fact that 
$$
|H^q_{\tor}(X_i(Y),\pa X_i(Y), \bE_{\bbZ})|\ge 1,
$$
 we see that
\begin{equation}
\begin{aligned}
\tau(X_i(Y),\bE,\bmu_{X_i})^{2(-1)^n} & \le \frac{\left(\underset{q +n\; \operatorname{odd}}{\prod} |H^q_{\tor}(X_i(Y),\pa X_i(Y);\bE_{\bbZ}|^2  \right)}{\left( \prod_q \vol_{\bmu^{\bbR}_{Z_i}}( H^{q-1}_{\free}(Z_i;\bE_{\bbZ}))^{(-1)^{q+1}} \right)^{(-1)^n}} \\
&=  \left( \prod_q \vol_{\bmu^{\bbR}_{Z_i}}( H^{q-1}_{\free}(Z_i;\bE_{\bbZ}))^{(-1)^{q+n}} \right) \left(\prod_{q+n \; \operatorname{odd}} |H^q_{\tor}(X_i(Y);\bE_{\bbZ}|^2  \right),
\end{aligned}
\label{gt.5}\end{equation}
where again $H^{q-1}_{\free}(Z_i;\bE_{\bbZ})$ must be replaced with $H^{q-1}_{\free}(Z_i;\bE_{\bbZ})\cap H_+^{q-1}(Z_i;\bE)$ when $q=n+1$.

Our objective is now to estimate the growth of the size of $H^*_{\tor}(X_i;\bE_{\bbZ})$ as $i\to \infty$ using this inequality, Theorem~\ref{pr.6} and \cite{MP2014}.  However, we need first to impose further restrictions on the sequence $\{\Gamma_i\}$ and request that it is cusp-uniform in the sense of \cite{Raimbault} and \cite[Definition~8.2 and Lemma~8.3]{MP2014}.  Since this notion will play an important role in what follows, let us briefly recall it.  Thus, let $P$ be a $\Gamma$-cuspidal parabolic subgroup of $G$.  For each $i$, let 
\begin{equation}
  \Lambda_{P}(\Gamma_i):= \log(\Gamma_i\cap N_P)
\label{gt.13}\end{equation}  
be the lattice in $\mathfrak{n}_P$ associated to $\Gamma_i\cap N_P$ and let 
\begin{equation}
   \Lambda^0_P(\Gamma_i):= (\vol(\Lambda_P(\Gamma_i)))^{-\frac{1}{2n}} \Lambda_P(\Gamma_i)
\label{gt.14}\end{equation}
be the corresponding unimodular lattice.  Let us also denote by $\mathcal{P}(\mathfrak{n}_P)$ the set of isometry classes of unimodular lattices in $\mathfrak{n}_P$ equipped with the standard topology induced by the natural identification
$$
         \mathcal{P}(\mathfrak{n}_P)\cong \SO(2n)\setminus \SL(2n,\bbR)/\SL(2n,\bbZ).
$$
We can now formulate the restriction we impose on the sequence $\{\Gamma_i\}$.  
\begin{definition}
The sequence $\{\Gamma_i\}$ is \textbf{cusp-uniform} in the sense that for each $\Gamma$-cuspidal parabolic subgroup $P$, there is a compact set $\cK_P$ in $\mathcal{P}(\mathfrak{n}_P)$ such that the lattice $\Lambda^0_P(\Gamma_i)$ is contained in  $\cK_P$ for each $i\in\bbN_0$.  
\label{gt.15}\end{definition}

We are now ready to give a proof of our main result Theorem~\ref{gt.30}.  
  
\begin{proof}[Proof of Theorem~\ref{gt.30}]
We know from \cite[Theorem~{1.1}]{MP2014} that 
\begin{equation}
   \lim_{i\to \infty} \frac{\log T(X_i;\bE,g_{X_i},h_E)}{[\Gamma: \Gamma_i]}= t^{(2)}_{\bbH^d}(\overline{\varrho})\vol(X)
   \label{gt.33}\end{equation}
   where $t^{(2)}_{\bbH^d}(\overline{\varrho})$ is the `local' logarithm of the $L^2$-torsion appearing in \cite[(1.1)]{MP2014}.  By \cite[Proposition~5.2]{BV}, we know that 
   $$
              (-1)^n t^{(2)}_{\bbH^d}(\overline{\varrho})>0.
   $$
   On the other hand, combining Proposition~\ref{gt.19} and Proposition~\ref{gt.28} with \eqref{gt.31}, we see that 
   \begin{equation}
   \lim_{i\to \infty} \frac{\sum_q (-1)^{q+n}\log \vol_{\bmu^{\bbR}_{Z_i}}( H^{q-1}_{\free}(Z_i;\bE_{\bbZ}))}{[\Gamma:\Gamma_i]}=0.
   \label{gt.34}\end{equation}
   Similarly, we see from  \eqref{gt.31} that 
\begin{equation}
   \lim_{i\to \infty} \frac{\kappa(X_i)  c_{\overline{\varrho}} }{[\Gamma: \Gamma_i]}=0.
\label{gt.36}\end{equation}
Moreover, by Poincar\'e duality, Milnor duality and \cite[Proposition~1.12]{Muller1993}, we have that 
\begin{equation}
  \log \tau(Z_i,\bE,\bmu_{Z_i})=0.
\label{gt.35}\end{equation}
Hence, equation \eqref{gt.32} follows by combining \eqref{gt.5} with Theorem~\ref{pr.6} together with \eqref{gt.33}, \eqref{gt.34}, \eqref{gt.36} and \eqref{gt.35}.

If we assume furthermore that $E_{\bbZ}^* \cong E_{\bbZ}$, then we have that 
\begin{equation}
    |H^q_{\tor}(X_i(Y);\bE_{\bbZ})| =  |H^q_{\tor}(X_i(Y);E_{\bbZ}|^2.
\label{gt.35b}\end{equation}
On the other hand, if we also assume that $\varrho$ is self-dual, which by \cite[\S~3.2.5]{GW} is automatic when $n$ is odd, we then know that there is a canonical isomorphism $E\cong E^*$ which is an isomorphism of flat vector bundles and of Hermitian vector bundles at the same time.  This implies in particular that 
\begin{equation}
      t^{(2)}_{\bbH^d}(\overline{\varrho})= 2 t^{(2)}_{\bbH^d}(\varrho).
\label{gt.35c}\end{equation}
Hence, combining \eqref{gt.32} with \eqref{gt.35b} and \eqref{gt.35c} gives \eqref{gt.32b}.
\end{proof}

\section{Applications in dimension 3} \label{d3.0}

If $n=1$ and $G=\Spin(3,1)\cong\SL(2,\bbC)$, we can find many situations where Theorem~\ref{gt.30} applies.  Recall that if $\varrho: G\to \GL(V)$ with $V=\bbC^2$ is the standard representation, then we can consider the $(m+1)$-dimensional complex representation 
\begin{equation}
\varrho_m: \SL(2,\bbC) \to \GL(V_m) \quad \mbox{with} \quad  V_m:= \operatorname{Symm}^m(V).
\label{sp.1}\end{equation} 
This representation is irreducible, both as a complex and a real representation.  Furthermore, if $(\varrho_m)_{\bbR}$ is $\varrho_m$ seen as a real representation, then for $m\in \bbN$, $(\rho_m)_{\bbR}$  gives a complete list of the finite dimensional real representations of $\SL(2,\bbC)$. Now, if we let $e_1=\left( \begin{array}{c}  1 \\ 0   \end{array} \right)$ and $e_2=\left( \begin{array}{c}  0 \\ 1   \end{array} \right)$ be the standard basis of $\bbC^2$, then 
$$
      v_j= e_1^j e_2^{m-j}, \quad j\in \{0,1,\ldots,m\},
$$
is a basis of $V_m$.  Using the natural non-degenerate complex linear pairing pairing   $\langle \cdot | \cdot\rangle: \bbC^2\otimes \bbC^2 \to \bbC$ defined by
\begin{equation}
            \langle  \left( \begin{array}{c}  a \\ b   \end{array} \right) | \left( \begin{array}{c}  c \\ d   \end{array} \right)\rangle:= ac + bd,
\label{gt.40}\end{equation}
we can identify $V^*$ with $V$ as a complex vector space,
\begin{equation}
          \sharp: V^* \to V.
\label{gt.41}\end{equation}
Under this identification, the basis $\{e_1,e_2\}$ is self-dual.  The identification \eqref{gt.41} also induces an identification 
\begin{equation}
\sharp: V_m^*\to V_m.
\label{gt.41b}\end{equation}  
Under this identification, the basis $\{v_j\}$ of $V_m$ is self-dual.  Under the identification \eqref{gt.41}, the dual representation $\varrho^*$ is simply given by 
$$
         \varrho(M)^*= (\varrho(M)^{-1})^T.  
$$
Thus, if $M= \left( \begin{array}{cc} a & b \\ c & d \end{array} \right)$, then a direct computation shows that 
\begin{equation}
      \varrho^*(M)= \left( \begin{array}{cc} d & -c \\ -b & a \end{array} \right)= \left( \begin{array}{cc} 0 & -1 \\ 1 & 0 \end{array} \right) \left( \begin{array}{cc} a & b \\ c & d \end{array} \right)\left( \begin{array}{cc} 0 & 1 \\ -1 & 0 \end{array} \right).
\label{gt.42}\end{equation}
In other words, making the identification \eqref{gt.41} as complex vector spaces, the element 
$$
A:= \left( \begin{array}{cc} 0 & -1 \\ 1 & 0 \end{array} \right)\in \SL(2,\bbC)
$$ 
induces an isomorphism of complex representations $A: V\to V^*$ such that $\varrho^*(M)= A\varrho(M) A^{-1}$.  Similarly, making the identification \eqref{gt.41b}, the element $A$  induces an isomorphism of complex representations 
\begin{equation}
   \varrho_m(A): V_m\to V_m^*,  \quad \varrho^*_m(M)= \varrho_m(A)\varrho_m(M)\varrho_m(A^{-1}).
\label{gt.43}\end{equation} 

Now, let $F:= \bbQ(\sqrt{-D})$ be an imaginary quadratic number field, where $D\in\bbN$ is square-free.  Let $\cO_D$ be its ring of integers, so that 
\begin{equation}
      \cO_D= \left\{  \begin{array}{ll}  \bbZ + \sqrt{-D}\bbZ,  & \mbox{if} \; D\equiv 1,2 \mod(4), \\
                                                              \bbZ+ \left( \frac{1+ \sqrt{-D}}2 \right)\bbZ, & \mbox{if} \; D\equiv 3 \mod(4).  \end{array}  \right.
\label{gt.44}\end{equation}
Let $\Gamma(D):= \SL(2,\cO_D)$ be the associated Bianchi group.  In $V_m$, a natural lattice preserved by $\Gamma(D)$ is given by 
\begin{equation}
      \Lambda_m:= \operatorname{Sym}^m \cO^2_D,
\label{gt.45}\end{equation}
so that, as a real vector space, $V_m= \Lambda_m\otimes_{\bbZ}\bbR$.  From this point of view, the representation $(\varrho_m)_{\bbR}$ descends to a representation over $\bbQ$ on 
$(V_m)_{\bbQ}:= \Lambda_m\otimes_{\bbZ} \bbQ$ obtained by restricting $\varrho_m$ to $\SL(2,\bbC)_{\bbQ}:= \SL(2,\bbQ(\sqrt{-D}))$.  If $\langle\cdot, \cdot\rangle: V_m^*\otimes V_m\to \bbC$ denotes the canonical non-degenerate pairing, then recall that the lattice $\Lambda_m^*\subset V_m^*$ dual to $\Lambda_m$ is given by 
$$
        \Lambda_m^*= \{  \eta\in V_m^* \; | \; \Re \langle \eta, \nu\rangle \in \bbZ \quad \forall \nu \in \Lambda_m\}.  
$$
Under the identification \eqref{gt.41b}, notice that we have a natural identification
\begin{equation}
\Lambda_m^*\cong \frac{\Lambda_m}{\delta_D} \quad \mbox{where} \quad \delta_D= \left\{  \begin{array}{ll}  \sqrt{-D},  & \mbox{if} \; D\equiv 1,2 \mod(4), \\
                                                               \frac{\sqrt{-D}}2, & \mbox{if} \; D\equiv 3 \mod(4).  \end{array}  \right.
\label{gt.45b}\end{equation}
Hence, the isomorphism of representations \eqref{gt.43} does not  identify $\Lambda_m$ with $\Lambda_m^*$, but the isomorphism of representations 
\begin{equation}
   \frac{\varrho_m(A)}{\delta_D}: V_m \to V_m^*
\label{gt.46}\end{equation}
does.  In fact, it induces an isomorphism of $\bbZ_{\Gamma(D)}$-modules, yielding the following result. 

\begin{lemma} 
Let $\Gamma\subset \Gamma(D)$ be any torsion-free subgroup so that $X=\Gamma\setminus \SL(2,\bbC)/\SU(2)$ is a finite volume hyperbolic manifold.  If 
$L_m= \bbH^3\times_{\left. \varrho_m\right|_{\Gamma}} \Lambda_m$ is the bundle of free $\bbZ$-modules associated to the lattice $\Lambda_m$, then there is a natural isomorphism $L^*_m\cong L_m$.  
\label{gt.47}\end{lemma}

Now, let $\ga$ be any nonzero ideal in $\cO_D$ and let $N(\ga)$ denote its norm.  Then the associated congruence subgroup $\Gamma(\ga)$ is defined by 
\begin{equation}
\Gamma(\ga):= \left\{ \left( \begin{array}{cc} a & b \\ c & d \end{array} \right)\in \Gamma(D) \; | \;  a-1\in \ga, \ d-1\in \ga, \ b,c\in \ga \right\},
\label{gt.48}\end{equation}
and the associated Hecke subgroup is defined by
\begin{equation}
\Gamma_0(\ga):= \left\{ \left( \begin{array}{cc} a & b \\ c & d \end{array} \right)\in \Gamma(D) \; | \;  \ c\in \ga \right\}.
\label{gt.49}\end{equation}
As explained in \cite[p.2779]{MP2014}, provided the norm $N(\ga)$ is sufficiently large, then $\Gamma(\ga)$ is torsion-free and \eqref{dc.3} holds, which allows us to prove Corollary~\ref{gt.50}.
\begin{proof}[Proof of Corollary~\ref{gt.50}]
The norm $N(\ga_i)$ is required to be sufficiently large to ensure that $\Gamma(\ga_i)$ is torsion-free and satisfies \eqref{dc.3}.  By Lemma~\ref{gt.47}, $L_m^*\cong L_m$, and by the discussion around \cite[Corollary~1.4]{MP2014}, the sequence $\Gamma(\ga_i)$ satisfies all the hypotheses of Theorem~\ref{gt.30}, so the result follows by noticing that 
$$
       |H^0_{\tor}(X_i(Y);L_{m})|=1
 $$
 thanks to the universal coefficient theorem.  Furthermore, when $m=2\ell$ is even, we know from \cite[(5.23)]{MP2013} that 
 $$
 -t^{(2)}_{\bbH^3}(\varrho_{2\ell})= \frac{\ell(\ell+1)+ \frac16}{\pi},
 $$
giving \eqref{int.7}.
\end{proof}

On the other hand, the Hecke subgroups $\Gamma_0(\ga)$ are never torsion-free and \eqref{dc.3} never holds.  However, taking a finite index subgroup $\Gamma'\subset \Gamma(D)$ which is torsion-free and for which \eqref{dc.3} holds, we may consider the subgroup
\begin{equation}
  \Gamma_0'(\ga):= \Gamma_0(\ga) \cap \Gamma',
\label{gt.51}\end{equation} 
which is torsion-free and satisfies \eqref{dc.3}. We can now prove Corollary~\ref{gt.52}. 
\begin{proof}[Proof of Corollary~\ref{gt.52}]
By Lemma~\ref{gt.47} and \cite[Theorem~1.5 and the surrounding discussion]{MP2014},  all the hypothesis of Theorem~\ref{gt.30} are satisfied and the result follows by noticing that 
$$
 |H^0_{\tor}(X_i(Y);L_{m})|=1
 $$
 thanks to the universal coefficient theorem. 
\end{proof}

For sequences of principle congruence subgroups, we can obtain an upper bound on the growth of torsion in cohomology using the approach of Raimbault \cite{Raimbault2013}.  We need first to establish the following two lemmas, \cf \cite[Lemma~6.5]{Raimbault2013} for the first lemma.

\begin{lemma}[Raimbault]  If $\mathfrak{a}$ is a non-zero ideal of $\cO_D$, then 
$$
       (m! \mathfrak{a}) V_m \subset (\Gamma(\mathfrak{a})-\Id)V_m.
$$
\label{ub.1}\end{lemma}
\begin{proof}
Consider the matrices 
$$
A_0:= \left(\begin{array}{cc} 0 & 0 \\ 1 & 0 \end{array}  \right) \quad \mbox{and} \quad A_{\infty}:= \left(\begin{array}{cc} 0 & 1 \\ 0 & 0 \end{array}\right),
$$
so that for all $a\in \mathfrak{a}$, 
$$
   \eta_a^0:= \Id + aA_0\in \Gamma(\mathfrak{a}), \quad \mbox{and} \quad \eta^{\infty}_a:=\Id + aA_{\infty} \in \Gamma(\mathfrak{a}).
$$
A direct computation shows that 
\begin{equation}
   (\eta^0_a-\Id)v_1= a v_0 \quad \mbox{and} \quad (\eta^{\infty}_a-\Id)v_{m-1}=av_m.
\label{ub.2}\end{equation}
More generally, one computes that 
$$
      (\eta^0_a-\Id)v_k= \sum_{j=0}^{k-1}  \left( \begin{array}{c} k \\ j \end{array} \right) a^{k-j}v_j,
$$
so that 
\begin{equation}
         kav_{k-1}= (\eta^0_a-\Id)v_k- \sum_{j=0}^{k-2}  \left( \begin{array}{c} k \\ j \end{array} \right) a^{k-j}v_j.
\label{ub.3}\end{equation}
Hence, using \eqref{ub.2} and \eqref{ub.3}, one can show by induction on $k$ that 
$$
     (k!) av_{k-1}\in (\Gamma(\mathfrak{a})-\Id)V_m \quad \forall k\in \{1,\ldots,m\},  \quad \forall a\in \mathfrak{a}.
$$
Since $av_m\in (\Gamma(\mathfrak{a})-\Id)V_m$ by \eqref{ub.2}, this completes the proof.  
\end{proof}

\begin{lemma}
If $\mathfrak{a}$ is a non-zero ideal of $\mathcal{O}_D$, then
$$
       N(\mathfrak{a})\le [\Gamma(D): \Gamma(\mathfrak{a})].
$$
\label{ei.1}\end{lemma}
\begin{proof}
Let $\displaystyle \mathfrak{a}= \prod_i \mathfrak{p}_i^{r_i}$,  $r_i\ge 1,$ be the factorization into prime ideals.  Then
$$
    N(\mathfrak{a})= \prod_i N(\mathfrak{p}_i)^{r_i}\ge \prod_{\mathfrak{p}|\mathfrak{a}} N(\mathfrak{p}),
$$
so that 
$$
  \prod_{\mathfrak{p}|\mathfrak{a}} \frac{1}{(1-\frac{1}{N(\mathfrak{p})^2})}= \prod_{\mathfrak{p}|\mathfrak{a}} \frac{N(\mathfrak{p})^2}{N(\mathfrak{p})^2-1}\le \prod_{\mathfrak{p}|\mathfrak{a}} N(\mathfrak{p})^2\le N(\mathfrak{a})^2.  
 $$
 Since 
 $$
      [\Gamma(D): \Gamma(\mathfrak{a})]= N(\mathfrak{a})^3 \prod_{\mathfrak{p}|\mathfrak{a}} \left( 1-\frac{1}{N(\mathfrak{p})^2} \right)
 $$
 by \cite[(11.1)]{MP2014}, the result follows.  
\end{proof}

\begin{theorem}(Raimbault \cite{Raimbault2013})\label{th-ub}
If $\mathfrak{a}_i$ is a sequence of non-zero ideals of $\cO_D$ such that $\displaystyle \lim_{i\to\infty} N(\mathfrak{a}_i)=\infty$, then for any $m\in \bbN$, for $X_D=\Gamma(D)\setminus \bbH^3$, $X_i=\Gamma(\mathfrak{a}_i)\setminus \bbH^3$ and for $L_m$ the bundle of free $\bbZ$-modules induced by $\varrho_m$ and the lattice $\Lambda_m$, we have that
\begin{equation}
\lim_{i\to \infty}\sup \frac{\log |H^2_{\tor}(\bX_i;L_m)|}{[\Gamma(D):\Gamma(\mathfrak{a}_i)]}\le -2t^{(2)}_{\bbH^3}(\varrho_m)\vol(X_D).
\label{ub.4b}\end{equation}
\label{ub.4}\end{theorem}
\begin{proof}
By Lemma~\ref{ub.1}, notice that 
\begin{equation}
      |H_0(\bX_i;L_m)|\le (m!N(\mathfrak{a}_i))^{m+1}.
\label{ubsr.1}\end{equation}
Hence, by Lemma~\ref{ei.1}, this means that 
$$
    \lim_{i\to\infty} \frac{\log |H_0(\bX_i;L_m)|}{[\Gamma(D):\Gamma(\mathfrak{a}_i)]}=0.
$$
By the universal coefficient theorem \eqref{gt.11}, the same is true for $H^1_{\tor}(\bX_i;L_m)$, namely
\begin{equation}
     \lim_{i\to\infty} \frac{\log |H_{\tor}^1(\bX_i;L_m)|}{[\Gamma(D):\Gamma(\mathfrak{a}_i)]}=0.
     \label{ub.5}\end{equation}
On the other hand, by Poincar\'e-Lefschetz duality,
$$
        H^3(\bX_i;L_m)\cong H_0(\bX_i,\pa\bX_i;L_m)=\{0\}.
$$      
Finally, we have trivially that $H^0_{\tor}(\bX_i;L_m)=\{0\}$.  Using \eqref{gt.3}, we see that in terms of relative cohomology,
\begin{equation}
      H_{\tor}^q(\bX_i,\pa\bX_i;L_m)\cong \{0\} \quad \mbox{for} \; q\in\{0,1\} \quad \mbox{and} \quad \lim_{i\to 0} \frac{\log |H_{\tor}^3(X_i;\pa X_i;L_m)|}{[\Gamma(D):\Gamma(\mathfrak{a}_i)]}=0.
\label{ub.6}\end{equation}     
These observations imply that \eqref{gt.2a} can be rewritten
$$
    \tau(\bX_i, \pa\bX_i;\bL_m)^{-2}=  \frac{|H^2_{\tor}(\bX_i,\pa\bX_i;\bL_m)|}{|H^3_{\tor}(\bX_i,\pa\bX_i;\bL_m)|}  \prod_q (\vol_{\bmu^{\bbR}_{X_i,\pa X_i}}(H^q_{\free}(\bX_i,\pa\bX_i,\bL_m)))^{(-1)^{q+1}}.
$$
Now, we can use \eqref{gt.4} as well as Proposition~\ref{gt.19} and Proposition~\ref{gt.28} to control the covolumes in this expression, so that combined with \eqref{ub.6}, this yields
$$
      \lim_{i\to \infty}\sup \frac{\log|H^2_{\tor}(\bX_i,\pa\bX_i;\bL_m)|}{[\Gamma(D);\Gamma(\mathfrak{a}_i)]}\le \lim_{i\to \infty}\frac{-2\log  \tau(\bX_i, \pa\bX_i;\bL_m)}{[\Gamma(D):\Gamma(\mathfrak{a}_i)]}.
$$
Hence, as in the proof of Theorem~\ref{gt.30}, we can use \eqref{gt.33} and Theorem~\ref{pr.6} to conclude that in fact
$$
\lim_{i\to \infty}\sup \frac{\log|H^2_{\tor}(\bX_i,\pa\bX_i;\bL_m)|}{[\Gamma(D);\Gamma(\mathfrak{a}_i)]}\le -2t^{(2)}_{\bbH^3}(\overline{\varrho}_m)\vol(X_D)>0.
$$
Using \eqref{gt.3}, this implies that 
$$
\lim_{i\to \infty}\sup \frac{\log|H^2_{\tor}(\bX_i;\bL_m)|}{[\Gamma(D);\Gamma(\mathfrak{a}_i)]}\le -2t^{(2)}_{\bbH^3}(\overline{\varrho}_m)\vol(X_D)>0.
$$
We can then use the fact that $L_m$ is self-dual to deduce the corresponding result for $L_m$.  
\end{proof}
Compared to \eqref{int.6}, the right hand side in \eqref{ub.4b} has an extra factor of $2$, so that Corollary~\ref{gt.50} and Theorem~\ref{ub.4} cannot be combined to conclude as in the compact case \cite{BV} that 
$$
\lim_{i\to \infty} \frac{\log|H^2_{\tor}(\bX_i;\bL_m)|}{[\Gamma(D);\Gamma(\mathfrak{a}_i)]}
$$
exists.  The precise point in our argument where this factor of $2$ appears is in the estimate \eqref{gt.2c}.  More precisely, an improvement in the upper bound in \eqref{gt.2c} would translate into an improvement of the lower bound in \eqref{gt.52b}, while an improvement in the lower bound of \eqref{gt.2c} would translate in an improvement of the upper bound in \eqref{ub.4b}.  Thus, to prove that the limit exists, one would need to improve estimate \eqref{gt.2c}.

\section{Exponential growth of torsion for sequences of representations} \label{sr.0}

Combining Theorem~\ref{pr.6} with the results of \cite{MP2012}, we can also obtain results about the growth of torsion when the group $\Gamma$ is fixed and the representation $\varrho$ varies.  To do so, we will restrict ourselves to the case where $\Gamma\subset \Gamma(D)= \SL(2,\cO_D)$ is a torsion-free finite index subgroup such that \eqref{dc.3} holds and consider the sequence of representations $\varrho_m$ defined in \eqref{sp.1}.  Let $X= \Gamma\setminus \SL(2,\bbC)/\SU(2)$ be the corresponding hyperbolic manifold of finite volume.  Denote by $E_m= \bbH^3\times_{\left. \varrho_m\right|_{\Gamma}} V_m$ the flat vector bundle associated to $\varrho_m$ with natural Hermitian metric $h_{E_m}$.  Recall also that we denote by $L_m= \bbH^3\times_{\left. \varrho_m\right|_{\Gamma}} \Lambda_m$ the associated bundle of free $\bbZ$-module, where $\Lambda_m= \Symm^m \cO^2_{D}$.     

If $\bbP^1(F)$ is the projective line of the field $F=\bbQ(\sqrt{-D})$, then we know, for instance from \cite{EGM}, that the cusps of $X$ are in bijection with 
$\Gamma\setminus \bbP^1(F)$.  Furthermore, in this setting, the Iwasawa decomposition $\SL(2,\bbC)= NAK$ is given by 
$$
   A= \left\{ \left( \begin{array}{cc}  \lambda & 0 \\ 0 & \lambda^{-1}  \end{array} \right) \; | \; \lambda>0 \right\},  \quad N= \left\{ \left( \begin{array}{cc}  1 & z \\ 0 & 1  \end{array} \right) \; | \; z\in\bbC \right\},
$$
and the standard parabolic subgroup is $P_0= NAM$ with 
$$
     M= \left\{ \left( \begin{array}{cc}  \lambda & 0 \\ 0 & \lambda^{-1}  \end{array} \right) \; | \; \lambda\in \bbS^1\subset \bbC \right\}.
$$
In particular, $P_0$ is the stabilizer of $\infty:= [1:0]\in \bbP(\bbC)$ and is a $\Gamma(D)$-cuspidal parabolic subgroup, hence a $\Gamma$-cuspidal parabolic subgroup as well.  

Given $P$ a $\Gamma$-cuspidal parabolic subgroup of $\SL(2,\bbC)$, it has a corresponding cusp represented by an element $\eta\in \bbP^1(F)$. Taking   
$$
B_P=\left( \begin{array}{cc}  a & b \\ c & d  \end{array} \right)\in \SL(2,F)
$$
such that $B_P\eta=\infty$, we see that $P$ is related to the standard parabolic subgroup via $P=B_P^{-1}P_0B_P$.  In particular, $N_P= B_P^{-1} N B_P$, which gives a corresponding identification of the Lie algebras, $\mathfrak{n}_P\cong \mathfrak{n}=\bbC$.  Under the identification $N_P\cong N$, the group $\Gamma_P$ becomes the following subgroup of 
$$
        B_P \Gamma(D)_P B_P^{-1}= \left\{ \left( \begin{array}{cc}  1 & \omega \\ 0 & 1  \end{array} \right) \; | \;  \omega \in \mathfrak{u}^{-2}\right\},
$$
where $\mathfrak{u}$ is the $\cO_D$-module generated by $c$ and $d$ and 
$$
     \mathfrak{u}^{-2} =\left\{  \omega\in F\; |  \; u\omega\in \cO_D  \;  \forall u \in \mathfrak{u}^2 \right\}.
$$
Thus, in terms of the identification $\mathfrak{n}_P\cong \mathfrak{n}=\bbC$, the lattice $\Lambda_P(\Gamma):= \log(\Gamma_P) \subset \mathfrak{n}_P$ is a sublattice of the lattice
\begin{equation}
       \Lambda_{P}(\Gamma(D))= \mathfrak{u}^{-2}\subset F\subset \bbC.
\label{sr.1}\end{equation}

With respect to this identification, the natural flat metric $g_{T_P}$ on $T_P= \Gamma_P\setminus N_P= \Lambda_P(\Gamma)\setminus \bbC$ is the one induced by the canonical Euclidean metric on $\bbC$.  On $T_P$, let $E_{m,P}= N_P\times_{\left. \varrho_m\right|_{\Gamma_P}} V_m$ be the flat vector bundle corresponding to $E_{m}$ with natural Hermitian metric $h_{E_{m,P}}$ and let $L_{m,P}= N_P\times_{\left. \varrho_m\right|_{\Gamma_P}} \Lambda_m$ be the corresponding bundle of free $\bbZ$-modules.  As in \S~\ref{co.0}, let us denote by $\vol(H^q_{\free}(T_P; L_{m,P}))$ the covolume of the lattice $H^q_{\free}(T_P; L_{m,P})$  in $H^{q}(T_P;E_{m,P})$ using the metric induced by Hodge theory and the metrics $g_{T_P}$ and $h_{E_{m.P}}$.  Similarly,
for $q=1$ let us denote by $\vol(H^1_{\free}(T_P; L_{m,P})\cap H^1_{\pm}(T_P;E_{m,P}))$ the covolume of the lattice 
$$
H^1_{\free}(T_P; L_{m,P})\cap H^1_{\pm}(T_P;E_{m,P}) \subset H^1_{\pm}(T_P;E_{m,P}).
$$  
We can now  state the  analog of Proposition~\ref{gt.28} for the sequence of representations $\{\varrho_m\}$.  
\begin{proposition}
There exists a positive constant $C_P$ depending on $P$ and $\Gamma$ such that for all $m\in \bbN$, 
$$
  \frac{1}{((m+1)!C_P^{m+1})^2} \le \frac{\vol(H^1_{\free}(T_P; \bL_{m,P})\cap H^1_{+}(T_P;\bL_{m,P}))}{\vol(H^0_{\free}(T_P; \bL_{m,P}))\vol(H^2_{\free}(T_P; \bL_{m,P}))}\le C_P
$$  
where $\bL_{m,P}= L_{m,P} \oplus L_{m,P}^*$ and $\bE_{m,P}= E_{m,P}\oplus E^*_{m,P}$.  
\label{sr.2}\end{proposition}
\begin{proof}
By Poincar\'e duality, notice that we have automatically that 
$$
\vol(H^0_{\free}(T_P; \bL_{m,P}))\vol(H^2_{\free}(T_P; \bL_{m,P}))=1,
$$
so we only need to prove that
\begin{equation}
  \frac{1}{((m+1)!C_P^{m+1})^2}    \le\vol(H^1_{\free}(T_P; \bL_{m,P})\cap H^1_{+}(T_P;\bE_{m,P}))\le  C_P.
\label{sr.2b}\end{equation}
Using \cite[(7.1)]{MR1} and  van Est's isomorphism induced by the map \eqref{dc.16}, we know that 
\begin{equation}
   \quad \cH^1_+(T_P; E_{m,P})= \bbC\iota_P(v_m\otimes d\bz), \quad \cH^1_-(T_P; E_{m,P})= \bbC\iota_P(v_0\otimes dz), 
\end{equation}
as well as the dual statement (using $\varrho^*_m$ instead of $\varrho_m$ to define $\iota_P$)
\begin{equation}
\cH^1_+(T_P; E^*_{m,P})= \bbC\iota_P(v^*_0\otimes d\bz), 
   \quad \cH^1_-(T_P; E^*_{m,P})= \bbC\iota_P(v^*_m\otimes dz).
\label{sr.3}\end{equation}
Let $\{\gamma_1,\gamma_2\}\subset F$ be a basis of $\Lambda_P(\Gamma)\subset F\subset \bbC\cong \mathfrak{n}_P$.  Let $\ell_i: [0,1] \to \Lambda_P(\Gamma)\setminus \mathfrak{n}_P$ be a smooth path with $\ell_i(0)=\ell_i(1)=0$ and with homotopy class $[\ell_i]\in \pi_1(T_P)=\pi_1(\Lambda_P(\Gamma)\setminus \mathfrak{n}_P)\cong \Lambda_{P}(\Gamma)$ representing $\gamma_i$.  Since $\iota_P(v_m)$ and $\iota_P(v_0^*)$ are flat sections, we compute that 
\begin{equation}
\begin{gathered}
       \int_{\ell_i} \iota_P(v_m\otimes d\bz)= \int_{\ell_i} e_1^m d\bz= \overline{\gamma}_i e_1^m  \in \Lambda_m\otimes_{\cO_D} F; \\
       \int_{\ell_i} \iota_P(v_0^*\otimes d\bz)= \int_{\ell_i} (e_2^*)^m d\bz= \overline{\gamma}_i (e_2^*)^m  \in \Lambda_m\otimes_{\cO_D} F.
       \end{gathered}       
\label{sr.4}\end{equation}
Thus, if $w_P\in \bbN$ is such that $w_P \overline{\gamma}_i \in \cO_D$ for $i\in\{1,2\}$, we see from \eqref{sr.4} that 
$$
  \{\nu_1,\nu_2\}= \left\{  \begin{array}{ll} \{  w_P \iota_P(v_m\otimes d\bz), \sqrt{-D} w_P \iota_P(v_m\otimes d\bz)\}, & \mbox{when} \; D\equiv 1,2 \mod(4) \\
      \{  w_P \iota_P(v_m\otimes d\bz), \frac{1 + \sqrt{-D}}{2} w_P \iota_P(v_m\otimes d\bz)\}, &  \mbox{when} \; D\equiv 3 \mod(4),
        \end{array}  \right.
$$         
is the basis of a \textbf{sublattice} $\cL_+(\varrho_m)$ of $H^1_{\free}(T_P; L_{m,P})\cap H^1_+(T_P;E_{m,P})$, so that 
\begin{equation}
\vol(H^1_{\free}(T_P; L_{m,P})\cap H^1_{+}(T_P; E_{m,P}))\le \vol(\cL_+(\varrho_m))= 2\vol(T_P) w_P^2 \Im \delta_D,
\label{sr.5}\end{equation} 
where $\delta_D$ is defined in \eqref{gt.45b}.  Similarly, 
$$
\{w_1,w_2\}= \left\{  \begin{array}{ll} \{  w_P \iota_P(v_0^*\otimes d\bz), \frac{1}{\sqrt{-D}} w_P \iota_P(v_0^*\otimes d\bz)\}, & \mbox{when} \; D\equiv 1,2 \mod(4), \\
 \{ 2w_P \iota_P(v_0^*\otimes d\bz), \left( 1 + \frac{1}{\sqrt{-D}}\right) w_P \iota_P(v_0^*\otimes d\bz)\}, &  \mbox{when} \; D\equiv 3 \mod(4), 
  \end{array}  \right.
$$    
is the basis of a \textbf{sublattice} $\cL_+(\varrho^*_m)$ of $H^1_{\free}(T_P; L^*_{m,P})\cap H^1_+(T_P;E^*_{m,P})$, so that 
\begin{equation}
\vol(H^1_{\free}(T_P; L^*_{m,P})\cap H^1_{+}(T_P; E^*_{m,P}))\le \vol(\cL_+(\varrho^*_m))= \frac{2\vol(T_P) w_P^2}{\Im \delta_D}.
\label{sr.5a}\end{equation} 
The upper bound in equation \eqref{sr.2b} then follows from \eqref{sr.4}, \eqref{sr.5}, and \eqref{sr.5a}.   

To obtain the lower bound, notice by \eqref{lb.45b} that it suffices to show that 
\begin{equation}
  \vol(H^1_{\free}(T_P; \bL_{m,P})\cap H^1_{-}(T_P; \bE_{m,P}))\le ((m+1)! C_P^{m+1})^2.
  \label{lb.1}\end{equation}
To show this, we perform the analog of \eqref{sr.4} and compute that
\begin{equation}
\begin{aligned}
       \int_{\ell_i} \iota_P(v_0\otimes d z) &= \int_{\ell_i} (\rho_m\left(\left(\begin{array}{cc} 1& z \\ 0 & 1  \end{array}\right)  \right)e_2^m) dz=  \int_{\ell_i}(e_2+z e_1)^{m}dz  \\
       &= \sum_{k=0}^m \left( \begin{array}{c} m \\ k  \end{array} \right)e_1^k e_2^{m-k}\frac{\gamma_i^{k+1}}{k+1} e_1^m  \in \Lambda_m\otimes_{\cO_D} F
       \end{aligned}       
\label{lb.2}\end{equation}
and
\begin{equation}
\begin{aligned}
       \int_{\ell_i} \iota_P(v_m^*\otimes d z) &= \int_{\ell_i} (\rho^*_m\left(\left(\begin{array}{cc} 1& z \\ 0 & 1  \end{array}\right)  \right)v_m^*) dz=  \int_{\ell_i}(e_1^*-z e_2^*)^{m}dz  \\
       &= \sum_{k=0}^m \left( \begin{array}{c} m \\ k  \end{array} \right)(-e_2^*)^k (e_1^*)^{m-k}\frac{\gamma_i^{k+1}}{k+1} e_1^m  \in \Lambda_m\otimes_{\cO_D} F.
       \end{aligned}       
\label{lb.3}\end{equation}
Thus, with $w_P\in \bbN$ as above so that $w_P\gamma_i\in \mathcal{O}_D$ for $i\in\{1,2\}$, we see from \eqref{lb.2} that 
$$
  \{  ((m+1)!w_P^{m+1}) \iota_P(v_0\otimes dz), \sqrt{-D} ((m+1)!w_P^{m+1}) \iota_P(v_0\otimes dz)\} \quad \mbox{if} \; D\equiv 1,2 \mod(4)
  $$   
 and 
$$
\{  ((m+1)!w_P^{m+1}) \iota_P(v_0\otimes dz), \frac{1 + \sqrt{-D}}{2} ((m+1)!w_P^{m+1}) \iota_P(v_0\otimes dz)\} \quad \mbox{if} \;D\equiv 3 \mod(4)
$$
is the basis of a \textbf{sublattice} $\cL_-(\varrho_m)$ of $H^1_{\free}(T_P; L_{m,P})\cap H^1_-(T_P;E_{m,P})$, so that 
\begin{equation}
\vol(H^1_{\free}(T_P; L_{m,P})\cap H^1_{-}(T_P; E_{m,P}))\le \vol(\cL_-(\varrho_m))= 2\vol(T_P) ((m+1)!w_P^{m+1})^2 \Im \delta_D.
\label{lb.4}\end{equation} 
Similarly, 
$$
 \{  ((m+1)!w_P^{m+1}) \iota_P(v_m^*\otimes d\bz), \frac{1}{\sqrt{-D}} ((m+1)!w_P^{m+1}) \iota_P(v_m^*\otimes d\bz)\} \quad \mbox{if} \; D\equiv 1,2 \mod(4)
$$
 and 
$$
\{ 2((m+1)!w_P^{m+1}) \iota_P(v_m^*\otimes d\bz), \left( 1 + \frac{1}{\sqrt{-D}}\right) ((m+1)!w_P^{m+1}) \iota_P(v_m^*\otimes d\bz)\} \quad \mbox{if} \; D\equiv 3 \mod(4)
$$
 is the basis of a \textbf{sublattice} $\cL_-(\varrho^*_m)$ of $H^1_{\free}(T_P; L^*_{m,P})\cap H^1_-(T_P;E^*_{m,P})$, so that 
\begin{equation}
\vol(H^1_{\free}(T_P; L^*_{m,P})\cap H^1_{-}(T_P; E^*_{m,P}))\le \vol(\cL_-(\varrho^*_m))= \frac{2\vol(T_P) ((m+1)!w_P^{m+1})^2}{\Im \delta_D}.
\label{lb.5}\end{equation} 
Combining \eqref{lb.4} and \eqref{lb.5} then gives the upper bound \eqref{lb.1} and completes the proof.
\end{proof} 

We can now combine Theorem~\ref{pr.6} with the results of \cite{MP2012} to prove Theorem~\ref{sr.6}.

\begin{proof}[Proof of Theorem~\ref{sr.6}]
By Proposition~\ref{sr.2}, equation \eqref{gt.5} and the fact that 
$$|H^0_{\tor}(X(Y);\bL_m)|=1,
$$
 there is a positive constant $C$ such that 
\begin{equation}
  \tau(X(Y),\bE_m,\bmu^m_{X})^{-2}\le ( (m+1)!C^{m+1})^2    |H^2_{\tor}(X(Y);\bL_m)|^2,
\label{sr.7}\end{equation}
where $\bmu^m_{X}$ is a basis of $H^*(X;\bE_m)$ as in \eqref{bc.1}. On the other hand, by Theorem~\ref{pr.6},
\begin{equation}
   \log T(X;\bE_m,g_X, h_{E_m})= \log \tau(X(Y),\bE_m,\bmu^m_X) + \kappa(X)c_{\varrho_m},
\label{sr.8}\end{equation} 
since $\tau(\pa X(Y), \bE_m, \bmu^m_{\pa X})=1$ by  Poincar\'e duality, Milnor duality and \cite[Proposition~1.12]{Muller1993}.  Furthermore, in \cite[Theorem~7.1]{MR1}, the constant $c_{\varrho_m}$ was explicitly computed to be
\begin{equation}
   c_{\varrho_m}= \log(m+1) + \frac{B(m)}{2} = \cO(m\log m) \quad \mbox{as} \; m\to \infty,
\label{sr.9}\end{equation}
where
$$
   B(m):=\sum_{\kappa=0}^{m-1}  \log\left( \frac{ \frac{m}2-\kappa +\sqrt{(\frac{m}2-\kappa)^2+2(1+\kappa)(m-\kappa)}}{\frac{m}2-\kappa-1+\sqrt{(\frac{m}2-\kappa-1)^2+2(1+\kappa)(m-\kappa)}} \right) .
$$
On the other hand, by \cite[Theorem~1.1]{MP2012},
\begin{equation}
  \lim_{m\to\infty} \frac{\log T(X;\bE_m,g_X,h_{\bE_m})}{m^2}= -\frac{\vol(X)}{\pi}.
\label{sr.10}\end{equation}
Hence, combining this equation with \eqref{sr.7}, \eqref{sr.8},  \eqref{sr.9} and Stirling's formula, we see that 
\begin{equation}
   \liminf_{m\to \infty} \frac{\log |H^2_{\tor}(X;\bL_m)|}{m^2}\ge \frac{\vol(X)}{\pi}.
\label{sr.11}\end{equation}
Finally, thanks to Lemma~\ref{gt.47}, $L^*_m\cong L_m$, so  
$$
|H^2_{\tor}(X;\bL_m)|= |H^2_{\tor}(X;L_m)|^2
$$
and the lower bound for the exponential growth of torsion follows by inserting this last equation in \eqref{sr.11}.

For the upper bound, we need first to control the growth of torsion in odd degrees.  Unfortunately, the upper bound \eqref{ubsr.1} coming from Lemma~\ref{ub.1} does not behave so well when $m\to \infty$. Instead we proceed as in \cite[Sect 4]{Marshall-Muller}. First note that
\begin{equation}\label{ubsr.2}
  H^3_{\tor}(X,\pa X;\bL_m)\cong  H_0(X;\bL_m)_{\tor}.
\end{equation}
Moreover $H_0(X;L_m)_{\tor}\cong (\Lambda_m)_{\Gamma}$, where the right
hand side denotes the co-invariants. For each prime $\pf$ of $F$, let
$V_{m,\pf}$ be the completion of $V_m$ at $\pf$, and let $\Lambda_{m,\pf}$ be the
completion of the image of $\Lambda_m$. Then $\Lambda_{m,\pf}$ is a
$\Gamma$-stable lattice, and we have
\begin{equation}\label{localiz}
  \log|(\Lambda_{m})_\Gamma|=\sum\log|(\Lambda_{m,\pf})_\Gamma|.
  \end{equation}
  Since $F$ is imaginary quadratic, \cite[Lemma 4.1]{Marshall-Muller} implies
  that $(\Lambda_{m,\pf})_{\Gamma}$ is trivial when $N(\pf)>m^2$. Hence we may
  assume that the summation on the right hand side of\eqref{localiz} runs
  over $N(\pf)\le m^2$ only. Using \cite[Proposition 4.2]{Marshall-Muller} and
  \eqref{localiz} it follows that there exists $C>0$ such that
\begin{equation}\label{bound-H0}
  \log|(\Lambda_{m})_\Gamma|\le \sum_{N(\pf)\le m^2}\frac{m\log N(\pf)}{N(\pf)}
  \le C m\log m
\end{equation}
for all $m\in\N$. Using \eqref{ubsr.2} and $|H_0(X;\bL_m)_{\tor}|
=|H_0(X;L_m)_{\tor}|^2=|(\Lambda_m)_\Gamma|^2$, we get  
\begin{equation}
    |H^3_{\tor}(X,\pa X;\bL_m)|=|H_0(X;\bL_m)_{\tor}|\le m^{Cm}.
    \label{ubsr.3}\end{equation}
Moreover, by \eqref{ub.6},
$$
       H^q_{\tor}(X,\pa X;\bL_m)\cong \{0\} \quad \mbox{for} \; q\in \{0,1\}.
$$
Using \eqref{gt.4} and Proposition~\ref{sr.2}, this means that \eqref{gt.2a} can be rewritten
\begin{equation}
\begin{aligned}
\tau(X,\pa X,\bL_m)^{-2} &= \frac{|H^2_{\tor}(X,\pa X;\bL_m)|}{|H^3_{\tor}(X,\pa X;\bL_m)|} \prod_q  \left( \vol_{\bmu^{\bbR}_{X,\pa X}} \left(H^q_{\free} (X,\pa X;\bL_{m}\right)\right)^{(-1)^{q+1}}  \\
 &\ge \frac{|H^2_{\tor}(X,\pa X;\bL_m)|}{|H^3_{\tor}(X,\pa X;\bL_m)|^2}  \frac{\vol(H^0_{\free}(\pa X; \bL_{m}))\vol(H^2_{\free}(\pa X; \bL_{m}))}{\vol(H^1_{\free}(\pa X; \bL_{m})\cap H^1_{+}(\pa X;\bL_{m}))}     \\
 &\ge  \frac{|H^2_{\tor}(X,\pa X;\bL_m)|}{|H^3_{\tor}(X,\pa X;\bL_m)|^2} \prod_{P\in \mathfrak{P}_{\Gamma}} \frac{1}{C_P}.
\end{aligned}
\label{ubsr.7}\end{equation}
In other words, we have that 
\begin{equation}
|H^2_{\tor}(X,\pa X;\bL_m)|\le \left( \prod_{P\in \mathfrak{P}_{\Gamma}} C_P \right) |H^3_{\tor}(X,\pa X;\bL_m)|^2\tau(X,\pa X,\bL_m)^{-2}.
\label{ubsr.4}\end{equation}
Hence, by \eqref{ubsr.2}, \eqref{sr.8}, \eqref{sr.9} and \eqref{sr.10}, we see that 
\begin{equation}
    \limsup_{m\to\infty} \frac{\log |H^2_{\tor}(X,\pa X;\bL_m)|}{m^2}\le \frac{2\vol(X)}{\pi}.
\label{ubsr.5}\end{equation}
Finally, since $H^2_{\tor}(X,\pa X;\bL_m)\cong H^2_{\tor}(X;\bL_m)$ by \eqref{gt.3} and $L^*_m\cong L_m$ by Lemma~\ref{gt.47}, we have that 
\begin{equation}
    |H^2_{\tor}(X,\pa X;\bL_m)|= |H^2_{\tor}(X;\bL_m)|^2,
\label{ubsr.6}\end{equation}
so the upper bound follows by substituting \eqref{ubsr.6} into \eqref{ubsr.5}.
\end{proof}

\bibliographystyle{amsalpha}
\bibliography{torsion_cusp}

\end{document}